\documentclass[11pt,reqno]{amsart}
\usepackage{amsmath, amsfonts, amsthm, amssymb, color}

\usepackage{setspace}
\usepackage{amsthm,amsmath,amssymb}
\usepackage{mathrsfs}
\usepackage{geometry}
\usepackage{amsfonts}
\usepackage[utf8]{inputenc}
\usepackage{amssymb}
\usepackage{amsmath}
\usepackage{subfigure}
\usepackage{graphicx}
\usepackage{amsmath,amscd}
\everymath{\displaystyle}
\usepackage{indentfirst} 
\usepackage{enumerate}
\usepackage[colorlinks]{hyperref}
\usepackage{bm}
\usepackage{algorithm}
\usepackage{algorithmic}
\usepackage{mathrsfs}
\usepackage{dsfont}
\usepackage[toc,page,title,titletoc,header]{appendix}
\newtheorem{prop}{Proposition}[section]
\newtheorem{rmk}{Remark}[section]
\newtheorem{defi}{Definition}[section]
\newtheorem{lemma}{Lemma}[section]
\newtheorem{thm}{Theorem}[section]

\newtheorem{cor}{Corollary}[section]
\usepackage{graphicx} 
\numberwithin{equation}{section}

\newcommand{\ud}{\,\mathrm{d}}
\newcommand{\udiv}{\, \mathrm{div}}

\newcommand{\R}{{\mathbb{R}}}
\newcommand{\red}{\textcolor{black}}

\author{Xuanrui Feng}
\address{School of Mathematical Sciences, Peking University, Beijing 100871, China.}
\email{pkufengxuanrui@stu.pku.edu.cn}

\author{Zhenfu Wang}
\address{Beijing International Center for Mathematical Research, Peking University, Beijing 100871, China}
\email{zwang@bicmr.pku.edu.cn}

\title[Propagation of Chaos]
{Quantitative Propagation of Chaos for 2D Viscous Vortex Model on the Whole Space}
\date{\today}


\date{\today}

\begin{document}
	
	\begin{abstract}
		 We derive the quantitative estimates of propagation of chaos for the large interacting particle systems in terms of the relative entropy between the joint law of the particles and the tensorized law of the mean field PDE. We resolve this problem for the first time for the viscous vortex model that \red{approximates} 2D Navier-Stokes equation in the vorticity formulation on the whole space. We obtain as key tools the Li-Yau-type estimates and Hamilton-type heat kernel estimates for 2D Navier-Stokes on the whole space.
	\end{abstract}

\maketitle

\tableofcontents

\section{Introduction}
Consider the evolution of $N$ indistinguishable particles given by the following stochastic differential equations (SDEs): 
\begin{equation}\label{particle}
    \ud X_t^i=\frac{1}{N} \sum_{j \neq i} K(X_t^i-X_t^j)\ud t +\sqrt{2\sigma} \ud W_t^i, \quad i=1, \cdots, N,
\end{equation}
where $X_t^i \in \mathbb{R}^2$ represents the position of particle $i$ at time $t$, and $\lbrace W_t^i \rbrace$ are $N$ independent standard Brownian motions on $\mathbb{R}^2$. The SDEs \red{reflect} that the states of particles are driven by the deterministic 2-body \red{interaction} kernel $K$ and the stochastic term $W_t^i$. We assume that $\sigma >0$.

In this article, we focus on the 2D viscous point vortex model on the whole space $\mathbb{R}^2$,  where the \red{interaction} kernel is given by the Biot-Savart law in fluid mechanics:
\begin{equation*}
    K(x)=\frac{1}{2\pi} \frac{(-x_2, x_1)}{|x|^2},
\end{equation*}
where $x=(x_1,x_2) \in \mathbb{R}^2$. Now the dynamics \eqref{particle} is closely connected with the famous 2D Navier-Stokes equation in the vorticity formulation:
\begin{equation}\label{limit}
    \partial_t w+(K \ast w) \cdot \nabla w= \sigma \Delta w.
\end{equation}

We introduce the \red{joint} law $\rho_N(t,x_1,\cdots,x_N)$ of the $N$-particle system \eqref{particle} and we are interested in the limiting behavior of $\rho_N$ as $N \rightarrow \infty$. Since those particles in \eqref{particle} are indistinguishable, we shall assume that the initial law $\rho_N(0)$ is a symmetric probability measure, denoted as $\rho_N(0) \in \mathcal{P}_{sym}(\mathbb{R}^{2N})$, and so is $\rho_N(t)$, for any $t >0$. By classical mean-field limit theory (see for instance the review \cite{jabin2017mean}), one expects that for any $i$, $(X_t^i)$ will converge to the limit McKean-Vlasov process:
\begin{equation}\label{mckeanvlasov}
    \ud \bar{X_t}=K \ast \bar{\rho_t}(\bar{X_t}) \ud t  +\sqrt{2\sigma} \ud  W_t, \quad \bar{\rho_t}=\mbox{Law}(\bar{X_t}),
\end{equation}
where $\bar{X_t} \in \mathbb{R}^2$ and $W_t$ is a standard Brownian motion on $\mathbb{R}^2$.  This process is self-consistent or distribution-dependent due to the appearance of the law of $\bar{X_t}$ in the SDE. By It\^o's formula, the law $\bar \rho_t$ of $\bar{X_t}$ satisfies the nonlinear Fokker-Planck equation:
\begin{equation}\label{truelimit}
    \partial_t \bar{\rho} +\udiv_x (\bar{\rho}(K \ast \bar{\rho}))=\sigma \Delta \bar{\rho},
\end{equation}
which is equivalent to \eqref{limit} by noticing that the Biot-Savart law is divergence-free, i.e. $\udiv K= 0$.  

By It\^o's formula, we write the $N$-particle Liouville equation solved by $\rho_N(t)$: 
\begin{equation}\label{liouville}
    \partial_t \rho_N+\sum_{i=1}^N \udiv_{x_i}  \Big(\rho_N (\frac{1}{N}\sum_{j \neq i} K(x_i-x_j)\Big) =\sigma \sum_{i=1}^N \Delta_{x_i} \rho_N.
\end{equation}
Here $\rho_N$ is understood as the entropy solution of \eqref{liouville}, the definition of which is given in \cite{jabin2018quantitative}. For the sake of completeness, we state it as follows.
\begin{defi}[Entropy Solution]\label{entropysolution}
    A density function $\rho_N \in L^1(\mathbb{R}^{2N})$, with $\rho_N \geq 0$ and $\int_{\mathbb{R}^{2N}} \rho_N \ud X^N=1$, is called an entropy solution to \eqref{liouville} in $[0,T]$, iff $\rho_N$ solves \eqref{liouville} in the sense of distributions and for a.e. $t \leq T$, 
    \begin{align*}
        &\int_{\mathbb{R}^{2N}} \rho_N(t,X^N) \log \rho_N(t,X^N) \ud X^N +\sigma \sum_{i=1}^N \int_0^t \int_{\mathbb{R}^{2N}} \frac{|\nabla_{x_i} \rho_N|^2}{\rho_N} \ud X^N \ud s\\
        \leq &\int_{\mathbb{R}^{2N}} \rho_N^0 \log \rho_N^0 \ud X^N.
    \end{align*}
Hereinafter $X^N= (x_1, \cdots, x_N) \in \mathbb{R}^{2N}$.    
\end{defi}
The existence of such entropy solutions to \eqref{liouville} has been proved in Proposition 1 in \cite{jabin2018quantitative}. Now it is convenient to compare the \red{joint} law $\rho_N$ to the following tensorized law
\begin{equation*}            
\bar{\rho}_N(t,x_1,\cdots, x_N)=\bar{\rho_t}^{\otimes N}(x_1, \cdots, x_N)=\prod_{i=1}^N \bar{\rho}(t, x_i).
\end{equation*}

We define the (scaled) relative entropy between the \red{joint} law $\rho_N$ and the $N$-tensorized law $\bar \rho^{\otimes^N }$ as follows.
\[
H_N (\rho_N \vert \bar \rho^{\otimes N} ) = \frac 1 N  \int_{\mathbb{R}^{2N}} \rho_N \log \frac{\rho_N }{\bar \rho^{\otimes N}} \ud X^N.  
\]

\subsection{Main Results}

Our main result is the development of the article \cite{jabin2018quantitative}, where Jabin and Wang applied the relative entropy method to obtain an explicit quantitative estimate of the convergence rate of system \eqref{particle} to its mean-field limit \eqref{truelimit}, while the state space domain there is restricted to $\mathbb{T}^2$. More precisely, Jabin and Wang obtained in \cite{jabin2018quantitative} the following relative entropy estimate:
\begin{equation*}
        H_N(\rho_N|\bar{\rho}_N)(t) \leq Me^{Mt}\, \Big(H_N(\rho_N^0|\bar{\rho}_N^0)+\frac{1}{N} \Big),
\end{equation*}
where $M$ is some universal constant depending on the initial data. The corresponding uniform-in-time propagation of chaos also for the torus case has been obtained recently in \cite{guillin2024uniform}. 

We are now able to extend the quantitative propagation of chaos result in \cite{jabin2018quantitative} to the whole space case by deriving a series of regularity results of Li-Yau-Hamilton-type and the Gaussian-type decay of the solution to the 2D Navier-Stokes equation \eqref{limit}. Our main result can be stated as follows. 
\begin{thm}[Entropic Propagation of Chaos]\label{entropybound}
    Assume that $\rho_N$ is an entropy solution to (\ref{liouville}) and that $\bar{\rho} \in L^\infty([0,T], L^1 \cap L^\infty (\mathbb{R}^2))$ solves (\ref{truelimit}) with $\bar{\rho} \geq 0$ and $\int_{\mathbb{R}^2} \bar{\rho}(t,x) \ud x =1$. Assume that the initial data $\bar{\rho}_0 \in W^{2,\infty}(\mathbb{R}^2)$ satisfies the growth conditions
    \begin{equation}\label{initialbound1}
        |\nabla \log \bar{\rho}_0(x)|^2 \leq C_1(1+|x|^2),
    \end{equation}
    \begin{equation}\label{initialbound2}
        |\nabla^2 \log \bar{\rho}_0(x)| \leq C_2(1+|x|^2),
    \end{equation}
    and the \red{Gaussian upper bound}
    \begin{equation}\label{initialbound3}
        \bar{\rho}_0(x) \leq C_3\exp(-C_3^{-1}|x|^2),
    \end{equation} 
    \red{for some positive constants $C_1, C_2, C_3$.} Then we have 
    \begin{equation*}
        H_N(\rho_N|\bar{\rho}_N)(t) \leq Me^{Mt^2}\Big( H_N(\rho_N^0|\bar{\rho}_N^0)+\frac{1}{N} \Big),
    \end{equation*}
    where $M$ is some constant that only depends on those initial bounds \red{$C_1,C_2,C_3$ and the initial norm $\|\bar\rho_0\|_{W^{2,\infty}}$}. 
\end{thm}

The significance of the entropy convergence is well studied now. In particular, once we define the $k$-marginal that is simply the law of the first $k$ particles $(X_t^1, \cdots, X_t^k)$ due to exchangeability, 
\begin{equation*}
    \rho_{N,k}(t,x_1,\cdots,x_k)=\int_{\mathbb{R}^{2(N-k)}} \rho_N(t,x_1,\cdots,x_N) \ud x_{k+1} \cdots \ud x_N,
\end{equation*}
we are able to control the (scaled) relative entropy between $\rho_{N,k}$ and the tensorized law $\bar{\rho}^{\otimes k}$ by the sub-additivity of relative entropy:
\begin{equation*}
    H_k(\rho_{N,k}|\bar{\rho}^{\otimes^k}) := \frac 1 k \int_{\mathbb{R}^{2 k} } \rho_{N, k} \log \frac{\rho_{N,k}}{\bar \rho^{\otimes^k}} \leq H_N(\rho_N|\bar{\rho}^{\otimes^N}).
\end{equation*}
Strong propagation of chaos then follows from the classical Csisz\'ar-Kullback-Pinsker inequality.
\begin{cor}[$L^1$-Propagation of chaos]
    Under the same assumptions of Theorem \ref{entropybound}, assume further that $H_N(\rho_N^0|\bar{\rho}_N^0) \rightarrow 0$ as $N \rightarrow \infty$, then for any $t \in [0,T]$, 
    \begin{equation*}
        H_N(\rho_N|\bar{\rho}_N)(t) \rightarrow 0, \,  \mbox{ as } N \rightarrow \infty.
    \end{equation*}
    Hence the $L^1$-propagation of chaos holds:
    \begin{equation*}
        \Vert \rho_{N,k} -\bar{\rho}^{\otimes^k} \Vert_{L^\infty([0,T],L^1(\mathbb{R}^{2k}))} \rightarrow 0.
    \end{equation*}
    Moreover, if $H_N(\rho_N^0|\bar{\rho}_N^0) \red{\leq \frac{C_0}{N}}$ \red{for some constant $C_0$}, then we obtain the convergence rate
    \begin{equation*}
        \Vert \rho_{N,k} -\bar{\rho}^{\otimes k} \Vert_{L^\infty([0,T],L^1(\mathbb{R}^{2k}))} \leq \frac{C_T}{\sqrt{N}}
    \end{equation*}
    \red{for some large constant $C_T$ that depends on $C_0,C_1,C_2,C_3,T,\|\bar\rho_0\|_{W^{2,\infty}}$.}
\end{cor}
\begin{rmk}\label{longtime}
	If we only care about the global convergence rate with respect to $N$ in terms of relative entropy, then the entropy bound in Theorem \ref{entropybound}, i.e. $H_N(\rho_N|\bar{\rho}_N)(t) \leq \frac{\red{M}e^{\red{M}t^2}}{N}$ given i.i.d. initial data $\bar \rho_N(0) = \bar \rho_0^{\otimes N}$ is optimal\red{-in-$N$}. 
	However, if the \red{interaction} kernel $K$ is bounded, for instance, one may expect to obtain a better convergence rate of marginals with respect to $N$ via the BBGKY hierarchy as done by Lacker \cite{lacker2021hierarchies}.  
    See also a recent mean-field limit result using similar approaches based on hierarchy in \cite{bresch2025new}. 
    It is still open \red{whether an optimal-in-$N$} convergence rate $\|\rho_{N, k}(t) - \bar \rho_t^{\otimes^k } \|_{L^1} \lesssim \frac{\red{k}}{N}$ holds for the vortex model with the \red{interaction} kernel given by the Biot-Savart law.  \red{For developments subsequent to the release of this manuscript, see Remark \ref{Added}.}
\end{rmk}
\begin{rmk}
	 Theorem \ref{entropybound} holds for any finite time horizon $[0, T]$. 
    We expect to extend this finite time convergence result to a uniform-in-time version as in \cite{guillin2024uniform}. This may rely on the dissipation properties of the limit PDE, which is here the vorticity formulation of the 2D Navier-Stokes equation, via some logarithmic Sobolev inequalities with Gaussian measures in $\mathbb{R}^2$ playing the role of the reference measure. We are content with the current propagation of chaos result and leave the possible uniform-in-time version for our future study. \red{For developments subsequent to the release of this manuscript, see Remark \ref{Added}.}
\end{rmk}

\red{
\begin{rmk}[Added in Proof] \label{Added}
  While this article was under review, several notable developments appeared in subsequent works. The sharp local propagation of chaos initiated by Lacker \cite{lacker2021hierarchies} has been extended to $\dot W^{-1,\infty}$ interaction kernels on the torus by S. Wang \cite{wang2024sharp}, under the additional assumption of high viscosity. This result essentially covers the periodic Biot-Savart kernel. For the whole space case, this result has been addressed in another work of the authors \cite{feng2024quantitative}. We have also generalized the model to the general circulation case and optimized the time dependence for the convergence of relative entropy, i.e. $H_N(\rho_N|\bar\rho_N) \lesssim \frac{(1+t)^M}{N}$. The work of Rosenzweig-Serfaty \cite{rosenzweig2024relative} considered the same problem in the general Riesz setting. A self-similar transformation has been applied to create a quadratic confinement in the mean-field dynamics. Under the assumption of uniform-in-time boundedness of the Hessian of the logarithm of the ratio of the solution relative to a Gaussian, a uniform-in-time propagation of chaos (actually stronger generation of chaos) is proved. This uniform-in-time bound has been later proved by Monmarch\'e-Ren-Wang \cite{monmarche2024time}.
\end{rmk}
}

\subsection{Gaussian Fluctuation}
Since we have established the relative entropy bound as in Theorem \ref{entropybound}, we can adapt the analysis in \cite{wang2023gaussian} to obtain a Gaussian fluctuation/Central Limit Theorem type result for the viscous vortex model on the whole space now.  We follow all the notation from the article \cite{wang2023gaussian}. We define the empirical measure as 
\[
\mu_N(t) = \frac 1 N \sum_{i=1}^N \delta_{X_t^i},
\]
and the fluctuation measure as 
\begin{equation}\label{fluc}
\eta^N_t = \sqrt{N} (\mu_N(t) - \bar \rho_t). 
\end{equation}
Propagation of chaos established in Theorem \ref{entropybound} implies that $\mu_N(t)$ converges weakly to $\bar \rho_t$. Gaussian fluctuation is about the next-order approximation of the empirical measure $\mu_N(t)$. Roughly speaking, from \eqref{fluc}, one has 
\[
\mu_N(t) = \bar \rho_t + \frac{1}{\sqrt{N}} \eta_N = \bar \rho_t + \frac{1}{\sqrt{N}} \eta_t + o\Big(\frac{1}{\sqrt{N}} \Big),  
\]
given that the asymptotic behavior (as $N\to \infty$) of the fluctuation measures $(\eta_t^N)$ can be described by a continuum model $(\eta_t)$.  

For the vortex model \eqref{particle} setting on the torus $\mathbb{T}^2$, both with $\sigma>0$ or $\sigma=0$, it has been shown in \cite{wang2023gaussian} that the fluctuation measures $(\eta_t^N)$ converges in distribution to a generalized Ornstein–Uhlenbeck process that can be described by a stochastic PDE. See Theorem 1.4 and Theorem 1.7 in \cite{wang2023gaussian} for details. 

When the viscous vortex model ($\sigma>0$) \eqref{particle} is set on the whole space $\mathbb{R}^2$, we still have the following Gaussian fluctuation type theorem. 

\begin{thm}[Gaussian Fluctuation] 
Under the same assumptions as in Theorem \ref{entropybound}, assume further that the sequence of initial fluctuations $\eta_0^N$ converges in distribution to $\eta_0$ in $\mathcal{S}'(\mathbb{R}^2)$, $H_N(\rho_N^0|\bar{\rho}_N^0) \leq \frac{C}{N}$ and $\bar \rho_0 \in W^{3, \infty}(\mathbb{R}^2)$.  Then the sequence of fluctuation measures $(\eta_t^N)_{t \in [0, T]}$ defined in \eqref{fluc} converges in distribution to $(\eta_t)_{t \in [0, T]}$ in the space $L^2 ([0, T], H^{- \alpha}) \cap C([0, T], H^{- \alpha -2})$, for every $\alpha >1$, where $(\eta_\cdot )$ is the unique martingale solution to the stochastic PDE:
	\[
	\partial_t \eta  = \sigma \Delta \eta  -  K \ast \eta \cdot \nabla \bar \rho -  K \ast \bar \rho \cdot \nabla \eta  - \sqrt{2 \sigma} \nabla \cdot (\sqrt{\bar \rho \xi }), \quad \eta(0) = \eta_0,
	\]
	where $\xi$ is a vector-valued space-time white noise on $\mathbb{R}^{+} \times \mathbb{R}^2$, i.e. a family of centered Gaussian random variables $\{ \xi (h) \vert h \in L^2 (\mathbb{R}^{+} \times \mathbb{R}^2, \mathbb{R}^2)\}$, such that 
	\[
	\mathbb{E} \big[ |\xi(h)|^2\big] = \|h\|^2_{L^2 (\mathbb{R}^+ \times \mathbb{R}^2, \mathbb{R}^2)}. 
 	\] 
\end{thm}
When the viscous vortex model \eqref{particle} is endowed with independent and identically distributed (i.i.d.) initial data, i.e. $\rho_N(0) = \bar \rho_0^{\otimes N}$, or more generally when initially the (not scaled) relative entropy  $\int_{\mathbb{R}^{2N}} \rho_N(0) \log \frac{\rho_N(0)}{\bar \rho_0^{\otimes N }} \to 0$ as $N \to \infty$ (see Proposition 5.2 in \cite{wang2023gaussian}), then the sequence of initial fluctuations $\eta_0^N $ converges in distribution to $\eta_0$, a Gaussian random variable, automatically. Furthermore, by Theorem \ref{entropybound}, the condition that $H_N(\rho_N^0|\bar{\rho}_N^0) \leq \frac{C}{N}$ implies that 
\[
\sup_{t \in [0, T]} H_N (\rho_N(t) \vert \bar \rho_t^{\otimes N}) \leq \frac{C_T}{N}, 
\]
where $C_T$ is a constant that only depends on $T$ and the initial data. This estimate plays a key role in \cite{wang2023gaussian} in establishing the Gaussian fluctuation result for $(\eta_t^N)_{t \in [0, T]}$. \red{In particular, combining with \cite[Lemma 2.6, Corollary 2.7]{wang2023gaussian}, we deduce that the initial data $\eta_0 \in H^{-\alpha}$, for $\alpha >1$.}   All other parts shall be standard which we leave for the interested readers to go through as in \cite{wang2023gaussian}. The remaining part of this article will then only concern propagation of chaos.   

\subsection{Related Literature}

Propagation of chaos and mean-field limit for the first-order system given in our canonical form \eqref{particle} have been extensively studied over the last decade. The basic idea of deriving some effective PDE describing the large scale behavior of interacting particle systems dates back to Maxwell and Boltzmann. But in our setting, the very first mathematical investigation can be traced back to McKean in \cite{mckean1967propagation}. See also the classical mean-field limit from Newton dynamics towards Vlasov kinetic PDEs in \cite{dobrushin1979vlasov,braun1977vlasov} and their recent development as in \cite{jabin2014review,bresch2025new}. Recently much progress has been made in the mean-field limit for systems as \eqref{particle} with singular interaction kernels, including those results focusing on the vortex model \cite{osada1986propagation,fournier2014propagation} and very recently on more general singular kernels as in \cite{duerinckx2016mean,jabin2018quantitative,serfaty2020mean,bresch2020mean,guillin2024uniform,nguyen2022mean,rosenzweig2022mean,rosenzweig2023global}. See also the references therein for more complete development on this subject. 

The point vortex approximation towards 2D Navier-Stokes/Euler equation arouses a lot of interest since the 1980s. Osada \cite{osada1986propagation} first obtained a \red{propagation} of chaos result for \eqref{particle} with a bounded initial distribution and a large viscosity $\sigma$. More recently, Fournier-Hauray-Mischler \cite{fournier2014propagation} obtained entropic propagation of chaos by the compactness argument and their result applies to all viscosities, as long as it is positive, and to all initial distributions with finite $k$-moment ($k>0$) and Boltzmann entropy. The vortex model \red{considered} in \cite{fournier2014propagation} is set on the entire plane $\mathbb{R}^2$, while no explicit convergence rate has been obtained there. A quantitative estimate of the propagation of chaos has been established in \cite{jabin2018quantitative} by evolving the relative entropy between the joint distribution of \eqref{particle} and the tensorized law at the limit. Quantitative propagation of chaos results in \cite{jabin2018quantitative} can cover first-order systems with a general class of \red{interaction} kernels, for instance those $K \in W^{-1, \infty}$ and $\udiv K \in W^{-1, \infty}$. However, the particle system \eqref{particle} lives in the compact domain $\mathbb{T}^d$. Theorem 1 in \cite{jabin2018quantitative} requires that $ \inf_{t \in [0, T] }\inf_{x \in \mathbb{R}^2} \bar \rho_t(x) >0$, which is not possible if the domain is the whole space. The validity of  Theorem 2 in \cite{jabin2018quantitative} is essentially not restricted to the torus case, though we still need some strong assumption as $\log \bar \rho$ is in some $BMO$ space. In this article, we firstly derive the quantitative propagation of chaos, which is optimal \red{in the number of particles $N$}, for the viscous vortex model set on the whole plane $\mathbb{R}^2$ that approximating the vorticity formulation of the 2D Navier-Stokes equation. 

The modulated energy method introduced in \cite{duerinckx2016mean,serfaty2020mean} can treat the mean-field limit problem set in the whole space $\mathbb{R}^d$ very well, in particular for the deterministic interacting particle systems as \eqref{particle} with $\sigma =0$.  We refer to the result in \cite{rosenzweig2022mean} for a mean-field convergence of point vortices to the 2D incompressible Euler equation with the only assumption that the vorticity is in $L^\infty$. For sub-Coulomb interactions, the authors in \cite{rosenzweig2023global} can even obtain global-in-time mean-field convergence for gradient/conservative diffusive flows. However, the results in \cite{rosenzweig2023global} cannot cover the situation of our viscous vortex model, even in a finite time horizon. 

\subsection{Structure of the Article}
The rest of this article is organized as follows. In Section \ref{section2} we first obtain some \textit{a priori} Sobolev-type regularity results of the limit PDE. Section \ref{section3} is devoted to the derivation of the Li-Yau-type gradient estimate for the 2D Navier-Stokes \red{equation}, and the parabolic Harnack-type inequality and the Gaussian-type decay from below to the solution are obtained as consequences. Section \ref{section4} is devoted to the Hamilton-type heat kernel estimate together with the Gaussian-type decay from above, which are the key estimates in proving the main result. In Section \ref{section5} we finish the proof of Theorem \ref{entropybound}. Section \ref{appA} and \ref{appB} are left for the elementary calculations for the proof of two technical lemmas for the Hamilton-type estimates in Section \ref{section4}.

\section{\textit{A Priori} Regularity Results}
\label{section2}
Our main result relies on the regularity of the limit equation (\ref{truelimit}). Consider the following 2D Navier-Stokes equation in the vorticity formulation
\begin{align}
    \label{vor}
    \partial_t \Bar{\rho} +(K \ast \Bar{\rho}) \cdot \nabla \Bar{\rho} &=\Delta \Bar{\rho},\\
    \Bar{\rho}(\cdot, 0)&=\bar{\rho}_0.
\end{align}
For simplicity we take $\sigma=1$. Let $u=K \ast \Bar{\rho}$ be the velocity field.

The well-posedness results and long-time behavior of the above Cauchy problem have been well studied in \cite{ben1994global,kato1994navier}. Given $\bar{\rho}_0 \in L^1(\mathbb{R}^2)$, one has $u(x,t) \in C^\infty(\mathbb{R}^2 \times \mathbb{R}^+)$, $\Bar{\rho}(x,t) \in C^\infty(\mathbb{R}^2 \times \mathbb{R}^+)$ and solves (\ref{vor}) in the classical sense. 

In our simplified setting, we also assume that $\Bar{\rho} \geq 0$ and that $\Bar{\rho} \in L^1 \cap L^\infty(\mathbb{R}^2)$ in order to keep $\bar{\rho}_0$ and thus $\Bar{\rho}(\cdot, t)$ a probability density function, whence the smoothness of $\Bar{\rho}$ is valid. We also assume on $\bar{\rho}_0$ the initial bounds (\ref{initialbound1})-(\ref{initialbound3}). These assumptions are consistent with those in our main result.

We are now ready to derive some Sobolev-type regularity results of $\Bar{\rho}$ directly from the PDE. Our main tool is the parabolic maximum principle. \red{Throughout this article, all the constants, unless explained specifically, are independent of time $t$ and only depend on the initial bounds $C_1,C_2,C_3,\|\bar\rho_0\|_{W^{2,\infty}}$ that appear in our main result. The constant $C$ may vary line-by-line, but all the other constants are fixed once being defined.}

\begin{lemma}\label{kconw} 
    Given  the initial data $\bar{\rho}_0 \in L^\infty \cap L^1(\mathbb{R}^2)$ for \eqref{vor}, one has that  for any $t \in [0, \infty)$, the velocity field $u_t=K\ast \Bar{\rho}_t \in L^\infty(\mathbb{R}^2)$ with
    \begin{equation}\label{ureg}
         \| u_t \|_{\infty} \leq \frac{\red{A_1}}{\sqrt{t \vee 1}},
    \end{equation}
    where $\Bar{\rho}_t (\cdot ) = \Bar{\rho}(\cdot, t)$ and $u_t(\cdot) = u(\cdot, t)$. \red{Here the constant $A_1$ depends only on $\|\bar\rho_0\|_\infty$.}
\end{lemma}
\begin{proof} \red{Write $\mathbf{1}_E$ for the indicator of a set $E$ (i.e., $\mathbf{1}_E(x)=1$ if $x\in E$ and $0$ otherwise). } Let $K_1 \red{(x)}=K\red{(x) \cdot \mathbf{1}}_{|x|\geq 1}$ and $K_2 \red{(x)}=K \red{(x) \cdot \mathbf{1}}_{|x|\leq 1}$.  Obviously $K_1 \in L^\infty$ and $K_2 \in L^1$. Hence by H\"older's inequality:
    \begin{equation*}
        \Vert K\ast \Bar{\rho}_t \Vert_\infty \leq \Vert K_1 \Vert_\infty \cdot \Vert \Bar{\rho}_t \Vert_1 +\Vert K_2 \Vert_1 \cdot \Vert \Bar{\rho}_t \Vert_\infty \leq C <\infty.
    \end{equation*}
Since $\|\Bar{\rho}_t\|_{1}=1$ is preserved over time and $\|\Bar{\rho}_t \|_{\infty} \leq \|\bar{\rho}_0 \|_{\infty}$, the above $C$ is uniform in time. The long-time decay rate is given in \cite[\red{(0.3)}]{kato1994navier}. \red{Therefore, there exists a constant $A_1$ satisfying \eqref{ureg}.}
\end{proof}

Notice that assumptions (\ref{initialbound1})-(\ref{initialbound3}) yield that $\bar{\rho}_0 \in W^{2,1} \cap W^{2,\infty}(\mathbb{R}^2)$.
\begin{lemma}\label{allsobolev}
    Given the initial data $\bar{\rho}_0 \in W^{2,1} \cap W^{2,\infty}(\mathbb{R}^2)$, one has that for any $t \in [0,\infty)$, the vorticity $\Bar{\rho}_t$ and the velocity field $u_t$ belong to $W^{2,1} \cap W^{2,\infty}(\R^2)$ with
    \begin{equation}\label{rhosobolev}
        \| \nabla^k \Bar{\rho} \|_{\infty} \leq \frac{\red{A_2}}{(t \vee 1)^{1+\frac{k}{2}}},
    \end{equation}
    \begin{equation}\label{usobolev}
        \| \nabla^k u \|_{\infty} \leq \frac{\red{A_2}}{(t \vee 1)^{\frac{1}{2}+\frac{k}{2}}},
    \end{equation}
    where $k=1,2$. \red{Here the constant $A_2$ depends only on $\|\bar\rho_0\|_{W^{2,\infty}}$.}
\end{lemma}
\begin{proof}
    According to Theorem B in \cite{ben1994global} and \cite{kato1994navier}, once the initial data $\bar{\rho}_0 \in L^1(\mathbb{R}^2)$, the Sobolev regularity and the long-time decay rate holds. Hence we only need to obtain \red{a} uniform Sobolev bound for a short time $t \in [0,t_0]$ for some small $t_0>0$.

    Denote by $\mathcal{L}$ the linear operator for a fixed $\Bar{\rho}$ that $\mathcal{L}=(K \ast \Bar{\rho}) \cdot \nabla$.
    We \red{derive the following equation solved by $\nabla \bar \rho$ from \eqref{vor}:}
    \begin{equation*}
        \partial_t(\nabla \Bar{\rho}) +\mathcal{L}(\nabla \Bar{\rho})= \Delta (\nabla \Bar{\rho}) -(K \ast \nabla \Bar{\rho}) \cdot \nabla \Bar{\rho}.
    \end{equation*}
    From the same decomposition of $K = K_1 + K_2$ as in Lemma \ref{kconw} and by H\"older's inequality:
    \begin{align*}
        \Vert K \ast \nabla \Bar{\rho} \Vert_\infty &\leq \Vert K_1 \ast \nabla \Bar{\rho} \Vert_\infty+\Vert K_2 \ast \nabla \Bar{\rho} \Vert_\infty \\
        &\leq \Vert K_1 \Vert_\infty \cdot \Vert \nabla \Bar{\rho} \Vert_1 +\Vert K_2 \Vert_1 \cdot \Vert \nabla \Bar{\rho} \Vert_\infty\\
        &\leq C(\Vert \nabla \Bar{\rho} \Vert_1+\Vert \nabla \Bar{\rho} \Vert_\infty).
    \end{align*}
    Hence by the  parabolic maximum principle, the evolution of $\Vert \nabla \Bar{\rho} \Vert_\infty$ is given by 
    \begin{equation*}
        \frac{\ud}{\ud t} \Vert \nabla \Bar{\rho} \Vert_\infty \leq C \Vert K \ast \nabla \Bar{\rho} \Vert_\infty \|\nabla \Bar{\rho}\|_{\infty} \leq C \Vert \nabla \Bar{\rho} \Vert_\infty (\Vert \nabla \Bar{\rho} \Vert_1+\Vert \nabla \Bar{\rho} \Vert_\infty),
    \end{equation*}
    while the evolution of $\Vert \nabla \Bar{\rho} \Vert_1$ is given by
    \begin{equation*}
        \frac{\ud}{\ud t} \Vert \nabla \Bar{\rho} \Vert_1 \leq C \Vert K \ast \nabla \Bar{\rho} \Vert_\infty \Vert \nabla \Bar{\rho} \Vert_1 \leq C \Vert \nabla \Bar{\rho} \Vert_1(\Vert \nabla \Bar{\rho} \Vert_1+\Vert \nabla \Bar{\rho} \Vert_\infty).
    \end{equation*}
    Summing up the two inequalities we obtain a closed inequality for $\Vert \nabla \Bar{\rho} \Vert_1+\Vert \nabla \Bar{\rho} \Vert_\infty$. Hence by Gr\"onwall's inequality these two norms for $\nabla \Bar{\rho}$ are finite for $t \leq t_0$ for some $t_0>0$. As a direct result $K \ast \nabla \Bar{\rho}=\nabla u \in L^\infty.$
    
    As for the second derivatives, we propagate $\nabla^2 \Bar{\rho}$ from (\ref{vor}) similarly that
    \begin{equation*}
        \partial_t(\nabla^2 \Bar{\rho})+\mathcal{L}(\nabla^2 \Bar{\rho})=\Delta(\nabla^2 \Bar{\rho})-2(K \ast \nabla \Bar{\rho}) \cdot \nabla^2 \Bar{\rho}-(K \ast \nabla^2 \Bar{\rho})\cdot \nabla \Bar{\rho}.
    \end{equation*}
    Notice that by H\"older's inequality:
    \begin{align*}
        \Vert K \ast \nabla^2 \Bar{\rho} \Vert_\infty &\leq \Vert K_1 \ast \nabla^2 \Bar{\rho} \Vert_\infty+\Vert K_2 \ast \nabla^2 \Bar{\rho} \Vert_\infty\\
        &\leq \Vert K_1 \Vert_\infty \cdot \Vert \nabla^2 \Bar{\rho} \Vert_1+ \Vert K_2 \Vert_1 \cdot \Vert \nabla^2 \Bar{\rho} \Vert_\infty\\
        &\leq C(\Vert \nabla^2 \Bar{\rho} \Vert_1+\Vert \nabla^2 \Bar{\rho} \Vert_\infty).
    \end{align*}
    Hence by the parabolic maximum principle, the evolution of $\Vert \nabla^2 \Bar{\rho} \Vert_\infty$ is given by
    \begin{equation*}
        \frac{\ud}{\ud t} \Vert \nabla^2 \Bar{\rho} \Vert_\infty \leq C \Vert 2(K \ast \nabla \Bar{\rho})\nabla^2 \Bar{\rho}+(K \ast \nabla^2 \Bar{\rho})\cdot \nabla \Bar{\rho} \Vert_\infty \leq C (\Vert \nabla^2 \Bar{\rho} \Vert_1+\Vert \nabla^2 \Bar{\rho} \Vert_\infty),
    \end{equation*}
    while the evolution of $\Vert \nabla^2 \Bar{\rho} \Vert_1$ is given by
    \begin{equation*}
        \frac{\ud}{\ud t} \Vert \nabla^2 \Bar{\rho} \Vert_1 \leq C \Vert 2(K \ast \nabla \Bar{\rho})\nabla^2 \Bar{\rho}+(K \ast \nabla^2 \Bar{\rho})\cdot \nabla \Bar{\rho} \Vert_1 \leq C (\Vert \nabla^2 \Bar{\rho} \Vert_1+\Vert \nabla^2 \Bar{\rho} \Vert_\infty).
    \end{equation*}
    Summing up the two inequalities we obtain a closed inequality for $\Vert \nabla^2 \Bar{\rho} \Vert_1+\Vert \nabla^2 \Bar{\rho} \Vert_\infty$. Hence by Gr\"onwall's inequality these two norms for $\nabla^2 \Bar{\rho}$ are finite for $t \leq t_0$ for some $t_0>0$. As a direct result $K \ast \nabla^2 \Bar{\rho}=\nabla^2 u \in L^\infty.$ \red{Therefore, there exists a constant $A_2$ satisfying \eqref{rhosobolev} and \eqref{usobolev}.}
\end{proof}

The above results show that $\Bar{\rho} \in W^{2,\infty}$ with a uniform Sobolev norm in $[0,\infty)$ that only depends on the norm of the initial data. Finally we conclude that
\begin{lemma}\label{kconwt}
    Under the same assumptions, one has for any $t \in [0,\infty)$, $K \ast \partial_t \Bar{\rho} \in L^\infty(\mathbb{R}^2)$ with
    \begin{equation}\label{partialtu}
        \| K \ast \partial_t \bar\rho \|_{\infty} \leq \frac{\red{A_3}}{(t \vee 1)^\frac{3}{2}}.
    \end{equation}
    \red{Here the constant $A_3$ depends only on $\|\bar\rho_0\|_{W^{2,\infty}}$.}
\end{lemma}
\begin{proof}
    Due to the Sobolev bounds obtained in Lemma \ref{allsobolev}, by (\ref{vor})
    \begin{align*}
        \Vert K \ast \partial_t \Bar{\rho} \Vert_\infty &\leq \Vert K \ast \Delta \Bar{\rho} \Vert_\infty + \Vert K \ast ((K \ast \Bar{\rho}) \cdot \nabla \Bar{\rho}) \Vert_\infty \\
        &\leq C+C \Vert K_1 \Vert_\infty \cdot \Vert (K \ast \Bar{\rho})\cdot \nabla \Bar{\rho} \Vert_1+C \Vert K_2 \Vert_1 \cdot \Vert (K\ast \Bar{\rho})\cdot \nabla \Bar{\rho} \Vert_\infty\\
        &\leq C(1+\Vert \nabla \Bar{\rho} \Vert_1 \cdot \Vert K \ast \Bar{\rho} \Vert_\infty+\Vert \nabla \Bar{\rho} \Vert_\infty \cdot \Vert K \ast \Bar{\rho} \Vert_\infty)\\
        &\leq C.
    \end{align*}
    The long-time decay rate is again given in \cite[\red{(0.5)}]{kato1994navier}. \red{Therefore, there exists a constant $A_3$ satisfying \eqref{partialtu}.} This completes the proof. 
\end{proof}
We should emphasize that all the results above are quite rough and only used to prepare for the following Li-Yau type estimates and Hamilton type estimates, which will give more detailed regularity results on the PDE. \red{These results are essentially contained in the literature of parabolic equations, see for instance \cite{giga2010nonlinear}. We collect them in this section for completeness and claim no originality of these results.}

\section{Li-Yau Type Gradient Estimates and Parabolic Harnack Inequality}
\label{section3}

In their celebrated paper \cite{Li1986OnTP}, Li and Yau derived the logarithmic gradient estimates of the solution of the heat equation, which led to a parabolic Harnack-type inequality. Now we are able to obtain these results for the vorticity function of the 2D Navier-Stokes equation. These results are enough to show that for any fixed time $t$, the solution $\Bar{\rho}(\cdot,t)$ satisfies a Gaussian-type decay from below on the whole space.

\begin{thm}[Li-Yau type Gradient Estimate]\label{liyau}
    Let $f=\log \Bar{\rho}$, where $\Bar{\rho}$ solves the 2D Navier-Stokes equation \eqref{vor}. Then \red{there exists a constant $A$ such that,} for $|x| \geq 2$, the following gradient estimate holds
    \begin{equation*}
        |\nabla f|^2-(1+|x|^{-2})\partial_t f \leq \red{A}(1+|x|^2),
    \end{equation*}
    while for $|x| \leq 2$ we have the following estimate
    \begin{equation*}
        |\nabla f|^2-\frac{5}{4}\partial_t f \leq \red{A}(1+|x|^2).
    \end{equation*}
    \red{Here the constant $A$ depends only on $\|\bar\rho_0\|_{W^{2,\infty}}$.}
\end{thm}

\begin{proof}
	For any $R\geq 2$, let $\alpha=1+R^{-2}$ and we denote by $F=|\nabla f|^2-\alpha \red{\partial_t f}$. The key idea is to estimate the maximum of $F$ in $B_R \times [0,T]$, where $B_R$ is the ball centered at $0$ with radius $R$. 
	 
	From (\ref{vor}) we obtain the propagation of $f(x,t)$ easily that
	\begin{equation*}
	\partial_t f+(K \ast \Bar{\rho}) \cdot \nabla f=\Delta f+|\nabla f|^2.
	\end{equation*}
	
	Hence we estimate $\Delta F$ by direct calculations that
	\begin{align*}
	\Delta F &=2|\nabla^2 f|^2 +2\nabla f \cdot \nabla(\Delta f)-\alpha \Delta \red{\partial_t f}\\
	&\geq |\Delta f|^2 +2\nabla f \cdot \nabla (\red{\partial_t f}+(K \ast \Bar{\rho}) \cdot \nabla f-|\nabla f|^2)-\alpha\red{\partial_t}(\red{\partial_t f}+(K \ast \Bar{\rho}) \cdot \nabla f-|\nabla f|^2)\\
	&=|\Delta f|^2+\red{\partial_t F}-2\nabla f \cdot \nabla F+\red{2\nabla(K \ast \bar\rho):(\nabla f \otimes \nabla f)}+(K\ast \Bar{\rho}) \cdot \nabla F-\alpha(K\ast \red{\partial_t}\Bar{\rho})\cdot \nabla f,
	\end{align*}
	where the inequality follows from the Cauchy-Schwarz inequality.
	
	In order to apply the parabolic maximum principle in a bounded domain, we introduce the cutoff function $\psi \in C^2[0,\infty)$ as done in \cite{Li1986OnTP} such that
	\begin{equation*}
	0 \leq \psi \leq 1, \quad \psi=1 \text{ on } [0,1], \quad \psi=0 \text{ on } [2,\infty),
	\end{equation*}
	and
	\begin{equation*}
	0 \geq \frac{\psi^{\prime}}{\psi^\frac{1}{2}} \geq -C, \quad \psi^{\prime \prime} \geq -C.
	\end{equation*}
	Let $\phi(x)=\psi\red{\Big(}\frac{|x|}{R}\red{\Big)}$. Hence $\phi F$ is supported in $B_{2R} \times [0,\infty)$ and equals to $F$ in $B_R \times [0,\infty)$.
	
	Now for fixed $T>0$, we assume that $\phi F$ attains at $(x_0,t_0)$ its maximum in $\mathbb{R}^2 \times [0,T]$. Without loss of generality, we may assume that $F(x_0,t_0) \geq 0$. Notice also that if $t_0=0$ then the maximum of $F$ in $B_R \times [0,T]$ is bounded by
 \begin{equation*}
     |\nabla \log \bar\rho_0|^2-\alpha \red{\partial_t \log \bar\rho \big|_{t=0}} \leq C(1+|x|^2).
 \end{equation*}
    Hence we may assume that $t_0>0$.
	
	By the parabolic maximum principle,
	\begin{equation*}
	\Delta(\phi F)(x_0,t_0) \leq 0, \quad \partial_t(\phi F)(x_0,t_0) \geq 0, \quad \nabla (\phi F)(x_0,t_0)=0.
	\end{equation*}
	Hence following the steps in \cite{Li1986OnTP}, at $(x_0,t_0)$ we have
	\begin{align*}
	0 \geq&\, \phi \Delta(\phi F)=\phi^2 \Delta F+2\phi \nabla \phi \cdot \nabla F+\phi F \Delta \phi\\
	\geq&\, -CR^{-2} (\phi F)+2\nabla \phi \cdot \nabla (\phi F)-2F|\nabla \phi|^2+\phi^2 |\Delta f|^2+\red{\partial_t (\phi(\phi F))}-2\phi^2\nabla f \cdot \nabla F\\
	&\, +\red{2\phi^2\nabla(K \ast \bar\rho):(\nabla f \otimes \nabla f)}+\phi^2(K \ast \Bar{\rho}) \cdot \nabla F-\alpha \phi^2(K \ast \red{\partial_t}\Bar{\rho}) \cdot \nabla f\\
	\geq&  -CR^{-2}(\phi F)+\phi^2 |\Delta f|^2+2\phi F \nabla f \cdot \nabla \phi-C\phi^2 |\nabla f|^2-\phi F(K \ast \Bar{\rho}) \cdot \nabla \phi -C\alpha \phi^2|\nabla f|,
	\end{align*}
	where we make use of Lemma \ref{allsobolev} and Lemma \ref{kconwt} to bound the terms with the \red{interaction} kernel $K$.
	
	Respectively, we have from Cauchy–Schwarz inequality
	\begin{equation*}
	2(\phi F)\nabla f \cdot \nabla \phi \geq -\epsilon_1 (\phi F)^2-\alpha \epsilon_1 \phi^2F\red{\partial_t f}-C\epsilon_1^{-1}R^{-2}\phi F,
	\end{equation*}
	\begin{equation*}
	-C\phi^2|\nabla f|^2-C\alpha \phi^2|\nabla f| \geq -C\phi^2 F-C\alpha \phi^2 \red{\partial_t f}-C\alpha \phi^2 -C\alpha \phi^2 F-C\alpha^2\phi^2 \red{\partial_t f},
	\end{equation*}
        \begin{equation*}
            -\phi F(K \ast \bar\rho) \cdot \nabla \phi \geq -CR^{-1} \phi^{\frac{3}{2}} F, 
        \end{equation*}
	where we make use of Lemma \ref{kconw} to bound the term with $K$.
	
	Also, from (\ref{vor}) we have
	\begin{align*}
	\phi^2 |\Delta f|^2 =&\phi^2(\red{\partial_t f}+(K\ast \Bar{\rho})\cdot \nabla f-|\nabla f|^2)^2\\
	=& \phi^2(F+(\alpha-1)\red{\partial_t f}-(K \ast \Bar{\rho})\cdot \nabla f)^2\\
	\geq& (\phi F)^2+2(\alpha-1)\phi^2 F\red{\partial_t f}-2((K \ast \Bar{\rho}) \cdot \nabla f) \phi^2F\\
	&+(\alpha-1)^2\phi^2\red{(\partial_t f)}^2+2(\alpha-1)\phi^2\red{\partial_t f}(K \ast \Bar{\rho})\cdot \nabla f\\
	\geq& \frac{1}{2}(\phi F)^2+2(\alpha-1)\phi^2F\red{\partial_t f}-C \phi^2F-C\alpha\phi^2\red{\partial_t f}+\frac{(\alpha-1)^2}{2}\phi^2\red{(\partial_t f)}^2.
	\end{align*}
	
	In order to cancel the term $\phi^2Ff_t$, we choose $\epsilon_1=\frac{2(\alpha-1)}{\alpha}$ and sum up all the terms.
	\begin{align*}
	0 \geq &\Big(\frac{2}{\alpha}-\frac{3}{2}\Big)(\phi F)^2-C\red{\Big(}1+R^{-1}+R^{-2}+\frac{\alpha}{2(\alpha-1)}R^{-2}\red{\Big)}(\phi F)\\
        &+\red{\Big(}\frac{(\alpha-1)^2}{2}\phi^2\red{(\partial_t f)}^2-C(\alpha+\alpha^2)\phi^2\red{\partial_t f}-C\alpha \phi^2\red{\Big)}\\
	\geq &\frac{C}{\alpha}(\phi F)^2-C(1+R^{-1}+R^{-2})(\phi F)-CR^4.
	\end{align*}
	By elementary calculations, $ax^2-bx-c \leq 0$ yields
	\begin{equation*}
	x \leq \frac{1}{2a}(b+\sqrt{b^2+4ac}) \leq \frac{b}{a}+\sqrt{\frac{c}{a}}
	\end{equation*}
	for $a,b,c>0$. Hence we sum up that
	\begin{equation*}
	\phi F \leq C\alpha(1+R^{-1}+R^{-2})+CR^2\sqrt{\alpha} \leq C(1+R^2).
	\end{equation*}
	Notice again that $\phi \equiv 1$ when $x \in B_R$. Hence we arrive actually at the following result.
	\begin{prop}\label{liyauesti}
		For any $R\geq 2$, there holds
		\begin{equation*}
		|\nabla f|^2-(1+R^{-2}) \partial_t f \leq C(1+R^2)
		\end{equation*}
		in $B_R \times [0,T]$.
	\end{prop}
	The last step is almost trivial. For $|x| \geq 2$ we take $R=|x|$, else we take $R=2$ and notice that $1+R^2 \leq C(1+|x|^2)$. \red{Therefore, there exists a constant $A$ satisfying the final result.}
\end{proof}

It is natural to derive from the Li-Yau-type estimate a parabolic Harnack-type inequality using the path integral.
\begin{cor}[Parabolic Harnack-type Inequality]\label{harnack}
Under the same assumptions as in Theorem \ref{liyau} and assume that $R \geq 2$. Then for any $x_1, x_2\in B_R$ and $t_1<t_2$,
    \begin{equation}\label{har}
        f(x_2,t_2)-f(x_1,t_1) \geq -C\frac{|x_1-x_2|^2}{t_2-t_1}-CR^2(t_2-t_1).
    \end{equation}
\end{cor}
\begin{proof}
    The proof is based on the method of proving the classical parabolic Harnack inequality, using integration along the straight segment connecting $(x_1,t_1)$ and $(x_2,t_2)$. By Proposition \ref{liyauesti},
    \begin{align*}
        f(x_2,t_2)-f(x_1,t_1) &=\int_0^1 \frac{\ud}{\ud s} f(sx_2+(1-s)x_1,st_2+(1-s)t_1) \ud s\\
        &=\int_0^1 \nabla f \cdot (x_2-x_1)+(t_2-t_1)\red{\partial_t f} \ud s\\
        &\geq \int_0^1 \nabla f \cdot (x_2-x_1)+(t_2-t_1)\frac{|\nabla f|^2-C(1+R^2)}{1+R^{-2}} \ud s\\
        &\geq -C\frac{|x_1-x_2|^2}{t_2-t_1}-CR^2(t_2-t_1), 
    \end{align*}
    where the last inequality is given by the Cauchy-Schwarz inequality, which then completes the proof. 
\end{proof}
It is direct to obtain a lower bound for the time derivative of $\log \Bar{\rho}$ from above.
\begin{cor}\label{timederibelow}
    $\partial_t \log \Bar{\rho}(x,t) \geq -C(1+|x|^2)$.
\end{cor}
\begin{proof}
    Let $x_1=x_2=x$ and $t_1=t$ in (\ref{har}). If $|x| \geq 2$, then we let $R=|x|$ and obtain
    \begin{equation*}
        f(x,t_2)-f(x,t) \geq -C|x|^2(t_2-t).
    \end{equation*}
    If $|x| \leq 2$,  then we let $R=2$ and obtain that 
    \begin{equation*}
        f(x,t_2)-f(x,t) \geq -C(t_2-t).
    \end{equation*}
    Letting  $t_2\rightarrow t$ gives the result.
\end{proof}

Furthermore, integrating over time will give the following Gaussian lower bound. 
\begin{cor}[Gaussian Lower Bound]\label{leastgauss}
    The solution to the vorticity formulation of the 2D Navier-Stokes equation has a Gaussian bound from below:
    \begin{equation*}
        \Bar{\rho}(x,t) \geq \Bar{\rho}(x,0)e^{-Ct(1+|x|^2)}.
    \end{equation*}
    \red{Here the constant $C$ depends only on $\|\bar\rho_0\|_{W^{2,\infty}}$.}
\end{cor}
Corollary \ref{leastgauss} exactly shows that the solution to the vorticity formulation must decay with at most Gaussian-type rate, given a fixed time, once the initial data satisfies this condition. This suggests that $\Bar{\rho}(x,t)$ should satisfy some regularity results similar to the Gaussian. 
This may agree to the idea in Remark \ref{longtime}, say $\Bar{\rho}_t$ could satisfy some logarithmic Sobolev inequality. \red{After the release of the preprint of this article, the idea of establishing an LSI for $\bar{\rho}_t$ by perturbing around the Gaussian was employed in \cite{rosenzweig2024relative}, which is the whole-space analogue of establishing an LSI by perturbing around the uniform measure, as done on the torus in \cite{guillin2024uniform}.
}


\red{Notice that by the mean-value inequality, the linear growth of $\nabla \log \bar\rho_0$ implies a Gaussian lower bound for $\bar\rho_0(x)$. Combining with the above corollary, we obtain the Gaussian lower bound of the vorticity function $\bar\rho(t,x)$.}
\begin{cor}[Gaussian-type decay I]\label{leastgauss2} Suppose that the initial vorticity $\bar{\rho}_0$ to the 2D Navier-Stokes equation \eqref{limit} satisfies \red{the growth condition \eqref{initialbound1}:
\begin{equation*}
    |\nabla \log \bar\rho_0(x)|^2 \leq C_1(1+|x|^2),
\end{equation*}
}
    then \red{we have a Gaussian} lower bound for any time $t$, i.e. \red{there exists a constant $c_1$ satisfying}
    \begin{equation*}
        \Bar{\rho}(x,t) \geq e^{-\red{c_1}(1+t)(1+|x|^2)}.
    \end{equation*}
\end{cor}

\red{
\begin{rmk}
    The Gaussian lower bound of the vorticity function has also been established in \cite{osada1987diffusion,giga1988two}. We present here another possible method by establishing the Li-Yau type estimate, which has its own interest for study.
\end{rmk}
}

\section{Hamilton Type Gradient Estimates}
\label{section4}
In the article \cite{hamilton1993matrix}, Hamilton derived another logarithm gradient estimate of the heat equation, together with a lower bound estimate of the Hessian matrix of the logarithm of the solution. Later an upper bound estimate of the Hessian is given in \cite{han2016upper}.

In the article \cite{li2014liyauhamilton}, Li \red{revisited the} notion of diffusion operator \red{\cite{bakry2013analysis}}, which is a natural extension of Laplacian on weighted manifolds. With the help of a related maximum principle, Li proved a logarithm gradient estimate of Hamilton-type and an upper bound for the Hessian matrix. Now similar results can be derived for the 2D Navier-Stokes equation, which are essential in the proof of our main results.

We recall that the velocity field $u= K \ast \Bar{\rho}$ and in particular $u$ can be viewed as the negative gradient vector field of the potential $f=g\ast \Bar{\rho}$, where $g(x)= - \frac{1}{2 \pi } \arctan \frac{x_1}{x_2}$. \red{In addition, we can define} $g(x)=\frac{1}{4}$ for $x_1<0, x_2=0$ and $g(x)=- \frac{1}{4}$ for $x_1>0, x_2=0$. Thus  $g \in L^\infty$ with  singularity only on the line $\{ (x_1, x_2) \vert x_2 = 0\}$. \red{This expression essentially coincides with the one in \cite{jabin2018quantitative}, namely $K=\nabla \cdot V$ for some $L^\infty$ diagonal matrix $V$. The diagonal elements coincide with $g(x)$ under some smooth correction for periodizing.} 

The weighted function $f$ has nice smoothness and regularity. In fact,  according to \cite{ben1994global}, given the initial data $\bar{\rho}_0 \in L^1$, then $\Bar{\rho}$ and the velocity field $u = K \ast \Bar{\rho}$ are smooth. Hence by $u = - \nabla f$ we conclude that $f$ is smooth.  


Now we define the diffusion operator 
\[
\Delta_f=\Delta+\langle \nabla f, \nabla \rangle, 
\]
with $u = K \ast \Bar{\rho} = - \nabla f$. Hence $\Bar{\rho}$ solves the equation $\partial_t \Bar{\rho}=\Delta_f \Bar{\rho}$. 

\subsection{Logarithm Gradient Estimate}

Our first result in this section is the following Hamilton-type logarithm gradient estimate.
\begin{thm}[Logarithm Gradient Estimate]\label{loggradient}
    Assume \red{as in our main result} that the initial data $\bar{\rho}_0 \red{\in W^{2,\infty}(\R^2)}$ satisfies the growth conditions
    \begin{equation}\label{nablaini}
        |\nabla \log \bar{\rho}_0(x)|^2 \leq \red{C_1}(1+|x|^2),
    \end{equation}
    \begin{equation}\label{deltaini}
        |\nabla^2 \log \bar{\rho}_0(x)| \leq \red{C_2}(1+|x|^2),
    \end{equation}
    and \red{the Gaussian upper bound}
    \begin{equation}\label{iniup}
        \bar{\rho}_0(x) \leq \red{C_3}\exp(-\red{C_3}^{-1}|x|^2),
    \end{equation}
    \red{for some constants $C_1,C_2,C_3$.} Then we have the linear growth control on the gradient of $\log \Bar{\rho}$:
    \begin{equation}\label{gradient}
        |\nabla \log \Bar{\rho}(x,t)| \leq \red{M_1}(1+|x|),
    \end{equation}
    \red{for some constant $M_1$ that depends on $C_1,C_2,C_3,\|\bar\rho_0\|_{W^{2,\infty}}$.}
\end{thm}
The proof of Theorem \ref{loggradient} relies on the following maximum principle associated to the diffusion operator, which was proved by Grigor'yan. We borrow the expression from \cite[Theorem 11.9]{grigoryan2009heat}.
\begin{thm}[Grigor'yan]\label{gri}
    Let $(M,g,e^f \ud V)$ be a complete weighted manifold, and let $F(x,t)$ be a solution of
    \begin{equation}\label{condigri1}
        \partial_t F = \Delta_f F \text{ in } M \times (0,T], \quad F(\cdot,0) = 0.
    \end{equation}
    Assume that for some $x_0 \in M$ and for all $r>0$,
    \begin{equation}\label{condigri2}
        \int_0^T \int_{B(x_0,r)} F^2(x,t) e^{f(x)} \ud V \ud t \leq e^{\alpha(r)}
    \end{equation}
    for some $\alpha(r)$ positive increasing function on $(0,\infty)$ such that
    \begin{equation}\label{condigri3}
        \int_0^\infty \frac{r}{\alpha(r)} \ud r=\infty.
    \end{equation}
    Then $F = 0$ on $M \times (0,T]$.
\end{thm}
\red{Here a complete weighted manifold $(M,g,e^f \ud V)$ is a complete Riemannian manifold $(M,g)$ associated with a weighted volume form $e^f \ud V$.} The above theorem is originally used to obtain the uniqueness of the bounded solution to the Cauchy problem, which helps to prove the stochastic completeness of weighted manifolds. We refer the proof of Theorem \ref{gri} and some related discussions to \cite{grigor2006heat,grigoryan2009heat}.

By examining the proof of Theorem \ref{gri}, we notice, as also shown in \cite{li2014liyauhamilton}, that the result has the following variant, which we will mainly refer to later. \red{In the following,  $F_+:= \max(F,0)$ denotes the positive part of $F$.}
\begin{thm}[Grigor'yan, an extended version]\label{karp}
    Let $(M,g,e^f \ud V)$ be a complete weighted manifold, and let $F(x,t)$ be a solution of
    \begin{equation}\label{condi1}
        \partial_t F \leq \Delta_f F \text{ in } M \times (0,T], \quad F(\cdot,0) \leq 0.
    \end{equation}
    Assume that for some $x_0 \in M$ and for all $r>0$,
    \begin{equation}\label{condi2}
        \int_0^T \int_{B(x_0,r)} F_+^2(x,t) e^{f(x)} \ud V \ud t \leq e^{\alpha(r)}
    \end{equation}
    for some $\alpha(r)$ positive increasing function on $(0,\infty)$ such that
    \begin{equation}\label{condi3}
        \int_0^\infty \frac{r}{\alpha(r)} \ud r=\infty.
    \end{equation}
    Then $F \leq 0$ on $M \times (0,T]$.
\end{thm}
For completeness, here we explain why Theorem \ref{karp} is valid. Indeed, since $\phi(x)=x_+$ is non-decreasing, convex, continuous and piecewise smooth, condition (\ref{condi1}) yields that
\begin{equation*}
    \partial_t F_+ \leq \Delta_f F_+ \text{ in } M \times (0,T], \quad F_+(\cdot,0) = 0.
\end{equation*}
Also in Theorem \ref{gri} the result still holds for $F$ satisfying $\partial_t F\leq \Delta_f F$ and $F \geq 0$, since the only part using the PDE is the equation (11.37) in \cite{grigoryan2009heat}, which turns into an inequality in the new case and still yields the following steps. Hence we apply Theorem \ref{gri} to $F_+$ and obtain Theorem \ref{karp}.

Now we turn to the proof of our main gradient estimates. In order to construct the auxiliary function $F(x,t)$ for our setting, we need some elementary calculations first.
\begin{lemma}\label{cal} 
Assume that $\Bar{\rho}$ solves the 2D Navier-Stokes equation \eqref{limit} and recall that $\Delta_f=\Delta+\langle \nabla f, \nabla \rangle$ with $u = - \nabla f = K \ast \Bar{\rho}$. Then
    \begin{align}
    &(\partial_t -\Delta_f) \red{\Big(}\frac{|\nabla \Bar{\rho}|^2}{\Bar{\rho}}\red{\Big)}=-\frac{2}{\Bar{\rho}}\red{\Big|}\nabla^2 \Bar{\rho}-\frac{\nabla \Bar{\rho} \otimes \nabla \Bar{\rho}}{\Bar{\rho}}\red{\Big|}^2-\frac{2}{\Bar{\rho}} \red{\nabla(K \ast \bar\rho):(\nabla \bar\rho \otimes \nabla \bar\rho)} \leq \frac{\red{2A_2}}{t \vee 1} \frac{|\nabla \Bar{\rho}|^2}{\Bar{\rho}},\label{nablawsquare}\\
    &(\partial_t-\Delta_f) (\Bar{\rho}\log \Bar{\rho})=-\frac{|\nabla \Bar{\rho}|^2}{\Bar{\rho}},\label{wlogw}
\end{align}
\end{lemma}

The proof is based on direct calculations, which we put into Section \ref{appA}.

Now we proceed to give the proof of Theorem \ref{loggradient}. 

\begin{proof}[Proof of Theorem \ref{loggradient}]
We define
\begin{equation*}
    F(x,t)=\frac{|\nabla \Bar{\rho}|^2}{\Bar{\rho}}+\red{B_1} \Bar{\rho}\log \Bar{\rho}-\red{B_2}\Bar{\rho}
\end{equation*}
for some \red{constants $B_1,B_2$ to be determined.} We deduce \red{from \eqref{nablawsquare} and \eqref{wlogw}} that
\begin{equation*}
    (\partial_t-\Delta_f)F \leq \red{2A_2}\frac{|\nabla \Bar{\rho}|^2}{\Bar{\rho}}-\red{B_1}\frac{|\nabla \Bar{\rho}|^2}{\Bar{\rho}} \red{\leq} 0,
\end{equation*}
\red{if we let $B_1 \geq 2A_2$.} Also, $F(\cdot, 0) \leq 0$ is equivalent to
\begin{equation*}
    |\nabla \log \bar{\rho}_0|^2+\red{B_1}\log \bar{\rho}_0 \leq \red{B_2}.
\end{equation*}
Thanks to the assumptions \eqref{nablaini} \eqref{iniup}, \red{the left-hand side is bounded above by
\begin{equation*}
    C_1(1+|x|^2)+B_1(\log C_3-C_3^{-1}|x|^2)=(C_1+B_1\log C_3)-(B_1 C_3^{-1}-C_1)|x|^2.
\end{equation*}
Therefore, we can choose $B_1 \geq C_1C_3$ and then choose $B_2 \geq C_1+B_1 \log C_3$ to ensure $F(\cdot,0) \leq 0$.
}

Now we are only left to check the assumption (\ref{condi2}). Recall that $\nabla \Bar{\rho} \in L^\infty$ and $\Bar{\rho} \in L^\infty$, and from Corollary \ref{leastgauss},
\begin{equation*}
    \Bar{\rho}(x,t) \geq \exp(-\red{c_1}(1+t)(1+|x|^2)). 
\end{equation*}
Consequently 
\begin{equation*}
    \int_0^T \int_{B_r} F_+^2(x,t) e^{f(x)} \ud V\ud t \leq Te^{C_T(1+r^2)} \int_{B_r} e^{f(x)} \ud x \leq CTr^2e^{C_T(1+r^2)}
\end{equation*}
since $f=g \ast \Bar{\rho}$ is bounded.
Hence we may choose $\alpha(r)=C_Tr^2(1+|\log r|)$ which satisfies (\ref{condi3}). Applying Theorem \ref{karp} we arrive at $F \leq 0$, or
\begin{equation}\label{F}
    |\nabla \log \Bar{\rho}|^2+\red{B_1}\log \Bar{\rho} \leq \red{B_2}(1+t).
\end{equation}
Recall that
\begin{equation*}
    \log \Bar{\rho} \geq -\red{c_1}(1+t)(1+|x|^2).
\end{equation*}
Substituting that into (\ref{F}) gives (\ref{gradient}) for short time.

In order to obtain the bound (\ref{gradient}) for long time, we need to construct another auxiliary function, say
\begin{equation*}
    F(x,t)=\phi \frac{|\nabla \bar\rho|^2}{\bar\rho}+\rho 
    \log \bar\rho -\red{B_3}\bar\rho,
\end{equation*}
for some $\phi(t)$ to be determined. We deduce \red{from \eqref{nablawsquare} and \eqref{wlogw}} that
\begin{equation*}
    (\partial_t-\Delta_f)F \leq \red{\Big(}\phi^\prime +\frac{\red{2A_2}}{t}\phi-1\red{\Big)} \frac{|\nabla \bar\rho|^2}{\bar\rho}.
\end{equation*}
Hence if we choose $\red{B_3}$ large and $\phi$ satisfying
\begin{equation*}
    \phi^\prime+\frac{\red{2A_2}}{t}\phi-1=0, \quad \phi(0)=0,
\end{equation*}
then the assumption (\ref{condi1}) holds immediately. Obviously $\phi(t)=\frac{t}{\red{2A_2}+1}$ works. The rest of assumptions of Theorem \ref{karp} follow as in the short time case. Hence we arrive at $F \leq 0$, or
\begin{equation}\label{FF}
    |\nabla \log \bar\rho|^2+\frac{\red{2A_2}+1}{t} \log \bar\rho \leq \frac{\red{B_3(2A_2+1)}}{t}.
\end{equation}
Recall that
\begin{equation*}
    \log \Bar{\rho} \geq -\red{c_1}(1+t)(1+|x|^2).
\end{equation*}
Substituting that into (\ref{FF}) gives (\ref{gradient}) for long time.
\end{proof} 

\subsection{Logarithm Hessian Estimate}

Following the previous steps, we can further derive the following logarithm Hessian estimate.
\begin{thm}[Logarithm Hessian Estimate]\label{loghessian}
    Assume that the initial data $\bar{\rho}_0$ satisfies the same conditions as in Theorem \ref{loggradient}, then we have the quadratic growth estimate on the Hessian of $\log \Bar{\rho}$, 
    \begin{equation*}
        |\nabla^2 \log \Bar{\rho}(x,t)| \leq \red{M_2}(1+|x|^2).
    \end{equation*}
\end{thm}
Notice that
\begin{equation*}
    \nabla^2 \log \Bar{\rho}=\frac{\nabla^2 \Bar{\rho}}{\Bar{\rho}}-\frac{\nabla \Bar{\rho} \otimes \nabla \Bar{\rho}}{\Bar{\rho}^2},
\end{equation*}
hence by Theorem \ref{loggradient} it suffices to control $\frac{\nabla^2 \Bar{\rho}}{\Bar{\rho}}$ by some quadratic function. We first give another lemma similar to Lemma \ref{cal}.
\begin{lemma}\label{cal2}  
	Assume that $\Bar{\rho}$ solves the 2D Navier-Stokes equation \eqref{limit} and recall that $\Delta_f=\Delta+\langle \nabla f, \nabla \rangle$ with $u = - \nabla f = K \ast \Bar{\rho}$. Then
    \begin{align}
        \label{hessianw}&(\partial_t -\Delta_f) \red{\Big(}\frac{|\nabla^2 \Bar{\rho}|^2}{\Bar{\rho}}\red{\Big)} \leq \frac{\red{5A_2}}{t \vee 1}\frac{|\nabla^2 \Bar\rho |^2}{\Bar{\rho}}+\frac{\red{A_2}}{(t \vee 1)^2}\frac{|\nabla \Bar{\rho}|^2}{\Bar{\rho}}.\\
        \label{wlogwsquare}&(\partial_t -\Delta_f) (\Bar{\rho}(\log \Bar{\rho})^2)=-\frac{2}{\Bar{\rho}}|\nabla \Bar{\rho}|^2(1+\log \Bar{\rho}).
    \end{align}
\end{lemma}
The proof is based on direct calculations, which we put into Section \ref{appB}.

We can proceed to prove Theorem \ref{loghessian}. 

\begin{proof}[Proof of Theorem \ref{loghessian}]
	
Firstly, we define
\begin{equation*}
    F(x,t)=e^{-\red{B_4}t}\frac{|\nabla^2 \Bar{\rho}|^2}{\Bar{\rho}}-\red{B_5}\bar{\rho}(\log \Bar{\rho})^2+\red{B_6} \Bar{\rho}\log \Bar{\rho}-\red{B_7} \Bar{\rho}
\end{equation*}
\red{for some constants $B_4,B_5,B_6,B_7$ to be determined}. We deduce \red{from \eqref{hessianw} and \eqref{wlogwsquare}} that
\begin{equation*}
    (\partial_t-\Delta_f)F \leq \red{5A_2}e^{-\red{B_4}t}\frac{|\nabla^2 \Bar{\rho}|^2}{\Bar{\rho}}+\red{A_2}e^{-\red{B_4}t}\frac{|\nabla \Bar{\rho}|^2}{\Bar{\rho}}-\red{B_4}e^{-\red{B_4}t}\frac{|\nabla^2 \Bar{\rho}|^2}{\Bar{\rho}}+\red{2B_5}\frac{|\nabla \Bar{\rho}|^2}{\Bar{\rho}}(1+\log \Bar{\rho})-\red{B_6}\frac{|\nabla \Bar{\rho}|^2}{\Bar{\rho}}\leq 0,
\end{equation*}
\red{if we let $B_4 \geq 5A_2$ and $B_6 \geq A_2+2B_5(1+\log_+ \|\bar\rho_0\|_\infty)$.} Also, $F(\cdot, 0) \leq 0$ is equivalent to
\begin{equation*}
    \frac{|\nabla^2 \bar{\rho}_0|^2}{\bar{\rho}_0^2}+\red{B_6}\log \bar{\rho}_0 \leq \red{B_7}+\red{B_5}(\log \bar{\rho}_0)^2.
\end{equation*}
\red{We recall the initial assumptions:
\begin{equation*}
    |\nabla \log \bar\rho_0(x)|^2 \leq C_1(1+|x|^2), \quad |\nabla^2 \log \bar\rho_0(x)| \leq C_2(1+|x|^2), \quad \bar\rho_0(x) \leq C_3 \exp(-C_3^{-1}|x|^2).
\end{equation*}
These allow us to choose $B_5$ large enough to eliminate the spatial growth in $|\nabla^2 \bar\rho_0|^2/\bar\rho_0^2$, and then choose $B_6$ large enough to satisfy $B_6 \geq A_2+2B_5(1+\log_+ \|\bar\rho_0\|_\infty)$ as previously shown, and finally choose $B_7$ large enough to make the above inequality valid.
}

Now we only need to check the assumption \eqref{condi2}. Recall that $\nabla^2 \Bar{\rho} \in L^\infty$, $\nabla \Bar{\rho} \in L^\infty$ and $\Bar{\rho} \in L^\infty$, and from Corollary \ref{leastgauss}:
\begin{equation}\label{leastgaussian2}
    \Bar{\rho}(x,t) \geq \exp(-\red{c_1}(1+t)(1+|x|^2)).
\end{equation}
Those imply  that 
\begin{equation*}
    \int_0^T \int_{B_r} F_+^2(x,t) e^{f(x)} \ud V\ud t \leq CTe^{C(1+T)(1+r^2)} \int_{B_r} e^{f(x)} \ud x \leq CTr^2e^{C_Tr^2}
\end{equation*}
since $f=g \ast \Bar{\rho}$ is bounded.
Hence we may choose $\alpha(r)=C_Tr^2(1+|\log r|)$ as well which satisfies \eqref{condi3}. Applying Theorem \ref{karp} we arrive at $F \leq 0$, or
\begin{equation}\label{F2}
    \frac{|\nabla^2 \Bar{\rho}|^2}{\Bar{\rho}^2} \leq e^{\red{B_4}T}( \red{B_5}(\log \Bar{\rho})^2-\red{B_6}\log \Bar{\rho}+\red{B_7}).
\end{equation}
Recall that
\begin{equation*}
    \log \Bar{\rho} \geq -\red{c_1}(1+t)(1+|x|^2).
\end{equation*}
\red{Also there exists a constant $c_2$ such that
\begin{equation*}
    \log \Bar{\rho} \leq c_2(1+t)(1+|x|^2)
\end{equation*}
which results from Lemma \ref{mostgauss} below.
}
Substituting those into \eqref{F2} gives the result for a short time.

For the long time case, we need to make use of the non-positive term which is abandoned in \eqref{nablawsquare} and construct another auxiliary function. Rewrite \eqref{nablawsquare} as
\begin{equation}
    (\partial_t-\Delta_f)\red{\Big(}\frac{|\nabla \bar\rho|^2}{\bar\rho}\red{\Big)} \leq -\frac{|\nabla^2 \bar\rho|^2}{\bar\rho}+2\frac{|\nabla \bar\rho|^4}{\bar\rho^3}+\frac{\red{2A_2}}{t \vee 1}\frac{|\nabla \bar\rho|^2}{\bar\rho},
\end{equation}
and define
\begin{equation*}
    F(x,t)=\phi \frac{|\nabla^2 \bar\rho|^2}{\bar\rho}+\psi \frac{|\nabla \bar\rho|^2}{\bar\rho}-\red{B_8}\bar\rho(\log \bar\rho)^2+\red{B_9}\bar\rho \log \bar\rho -\red{B_{10}}\bar\rho,
\end{equation*}
for some auxiliary functions $\phi(t)$ and $\psi(t)$ \red{and constants $B_8,B_9,B_{10}$} to be determined. We deduce \red{from \eqref{hessianw} and \eqref{wlogwsquare} }that
\begin{align*}
    (\partial_t-\Delta_f)F \leq &\, \red{\Big(}\phi^\prime+\frac{\red{5A_2}}{t}\phi-\psi\red{\Big)}\frac{|\nabla^2 \bar\rho|^2}{\bar\rho}\\
    +&\, \red{\Big(}\frac{\red{A_2}}{t^2}\phi+\psi^\prime+\frac{\red{2A_2}}{t}\psi+2\psi|\nabla \log \bar\rho|^2+2\red{B_8}\log \bar\rho+2\red{B_8}-\red{B_9}\red{\Big)}\frac{|\nabla \bar\rho|^2}{\bar\rho}.
\end{align*}
Also we set
\begin{equation*}
    \phi(0)=\psi(0)=0.
\end{equation*}
Recall the proof of Theorem \ref{loggradient}, we obtain for $\psi(t)=\frac{\red{B_8}t}{\red{2A_2}+1}$ that
\begin{equation*}
    2\psi|\nabla \log \bar\rho|^2+2\red{B_8}\log \bar\rho \leq \red{2B_8B_3}.
\end{equation*}
Now letting $\phi(t)=\frac{\red{B_8}t^2}{\red{(2A_2+1)(5A_2+2)}}$, we are able to cancel the first term on the right-hand side. Pick $\red{B_9}$ large to cancel the second term. The initial condition is also satisfied for large $\red{B_{10}}$. The assumption \eqref{condi3} is verified similarly as in the short time case.

Hence we arrive at $F \leq 0$, or
\begin{equation}\label{t2}
    t^2 \frac{|\nabla^2 \bar\rho|^2}{\bar\rho^2} \leq \red{B}(\log \bar\rho)^2+\red{B}
\end{equation}
\red{for some large constant $B$.} Recall that
\begin{equation*}
    \log \Bar{\rho} \geq -\red{c_1}(1+t)(1+|x|^2) \text{ and } \log \Bar{\rho} \leq \red{c_2}(1+t)(1+|x|^2).
\end{equation*}
Substituting those into \eqref{t2} gives the result for the long time.
\end{proof}

Finally we present the Gaussian-type decay from above of the solution.
\begin{lemma}[Gaussian-type decay II]\label{mostgauss}
    Under the same assumptions of the initial vorticity $\bar{\rho}_0$ as in Theorem \ref{loggradient}, one has the Gaussian upper bound for $\Bar{\rho}(x,t)$ for any time $t$, i.e. \red{there exists a constant $C_2'$ such that}
    \begin{equation*}
        \Bar{\rho}(x,t) \leq \frac{\red{C_2'}}{t \vee 1} \exp\red{\Big(}-\frac{|x|^2}{8t+\red{C_2'}}\red{\Big)}.
    \end{equation*}
    \red{Here the constant $C_2'$ depends on $C_3$.}
\end{lemma}
\begin{proof}
    This estimate is a direct combination of the initial condition and the pointwise estimate given in \cite[Theorem 3]{carlen1996optimal}, \red{from which we have for some universal constant $C$,}
    \begin{equation}
        \bar\rho(x,t) \leq C \int \frac{1}{t} \exp\red{\Big(}-\frac{|x-y|^2}{8t}\red{\Big)} \bar\rho_0(y) \ud y.
    \end{equation}
    By the initial condition \eqref{iniup}
    \begin{equation*}
        \bar\rho_0(y) \leq \red{C_3}\exp(-\red{C_3}^{-1}|y|^2),
    \end{equation*}
    we calculate that
    \begin{align*}
        \bar\rho(x,t) &\leq C\red{C_3} \int \frac{1}{t} \exp\red{\Big(}-\frac{\red{C_3}|x-y|^2+8t|y|^2}{8t\red{C_3}} \red{\Big)} \ud y\\
        &\leq C\red{C_3}\int \frac{1}{t} e^{-\frac{|x|^2}{8t+\red{C_3}}} \exp \red{\Big(}-\frac{8t+\red{C_3}}{8t\red{C_3}}|y|^2 \red{\Big)} \ud y\\
        &\leq \frac{\red{C_2'}}{t \vee 1} \exp\red{\Big(}-\frac{|x|^2}{8t+\red{C_2'}} \red{\Big)},
    \end{align*}
    \red{for some constant $C_2'$.}
\end{proof}

\section{Proof of the Main Result}
\label{section5}
Now we are ready to give the proof of our main result. The proof mainly follows from the idea of \cite{jabin2018quantitative}, controlling the time derivative of relative entropy by the relative entropy itself and some extra small terms. The result comes from the classical Gr\"onwall argument.

Firstly, adapting the proof of Lemma 2 in \cite{jabin2018quantitative}, simply changing the domain from the torus $\mathbb{T}^d$ to the whole space $\mathbb{R}^2$,  we obtain the time evolution of the relative entropy. Hereinafter we may use the convention that $K(0)=0$.

\begin{lemma}\label{propagation} Assume that $\rho_N$ is an entropy solution as per Definition \ref{entropysolution}. Assume that $\bar \rho \in W^{2, \infty } ([0, T] \times \mathbb{R}^2)$ solves the limit PDE \eqref{limit} with $\int_{\mathbb{R}^2} \bar \rho(t, x) \ud x =1$ and $\bar \rho(t, x) \geq 0$. Then we have:
\begin{equation*}
    \begin{split}
         & H_N(\rho_N|\bar{\rho}_N)(t)  = \frac 1 N \int_{\mathbb{R}^{2N}} \rho_N(t, X^N) \log \frac{\rho_N(t, X^N)}{\bar \rho_N (t, X^N)} \ud X^N  \\ & \leq H_N(\rho_N^0 \vert \bar \rho_N^0)  -\frac{1}{N^2} \sum_{i,j=1}^N \int_0^t  \int_{\mathbb{R}^{2N}} \rho_N \Big(K(x_i-x_j)-K \ast \bar{\rho}(x_i) \Big) \cdot \nabla \log \bar{\rho}(x_i) \ud X^N \ud s \\
         & \qquad  \qquad \qquad  - \frac{\sigma}{N} \sum_{i=1}^N \int_0^t \int_{\mathbb{R}^{2N}} \rho_N \red{\Big|}\nabla_{x_i} \log \frac{\rho_N}{ \bar \rho_N}\red{\Big|}^2 \ud X^N \ud s, 
         \end{split}
    \end{equation*}
 where we recall that $\bar \rho_N(t, X^N) = \prod_{i=1}^N \bar \rho(t, x_i)$.   
\end{lemma}


The last term is currently abandoned since it is obviously non-positive, but we mention that, as in \cite{guillin2024uniform}, one may make use of this term by treating it as the Fisher information and applying the logarithmic Sobolev inequality. This may produce an extra negative term of time integral of the relative entropy.

The fact that $\bar\rho_N$ satisfies some LSI can be derived by the tensor property and directly examining the McKean-Vlasov SDE \eqref{mckeanvlasov}, but the LSI term is not large enough to overwhelm the other relative entropy term produced by the following steps. We leave the uniform-in-time result in our future study.

\red{After the submission of this article, such an LSI for the solution to the vorticity equation on the whole space with quadratic confinement has been established in \cite{monmarche2024time}.}

Now since $K$ is odd, we can adapt the classical symmetrization trick (see the proof in Theorem 2 of \cite{jabin2018quantitative}) to obtain the following inequality. 
\[
\begin{split}
  &\, H_N(\rho_N|\bar{\rho}_N)(t)  \\   \leq  &\,  H_N(\rho_N^0 \vert \bar \rho_N^0)      -  \frac{1}{N^2} \sum_{i,j=1}^N \int_0^t  \int_{\mathbb{R}^{2N}} \rho_N \Big(K(x_i-x_j)-K \ast \bar{\rho}(x_i) \Big) \cdot \nabla \log \bar{\rho}(x_i) \ud X^N \ud s \\ 
  = &\,   H_N(\rho_N^0 \vert \bar \rho_N^0)    - \int_0^t  \int_{\mathbb{R}^{2N}} \rho_N \Big( \frac{1}{N^2} \sum_{i, j =1}^N \phi (x_i, x_j)   \Big)  \ud X^N \ud s,  \\
\end{split}
\]
where the function $\phi$ is defined as 
\begin{equation}\label{phi}
    \phi(x,y)= \frac 1 2  K \ast \bar{\rho} (x) \cdot \nabla \log \bar{\rho}(x)+ \frac 1 2 K \ast \bar{\rho} (y) \cdot \nabla \log \bar{\rho}(y)- \frac 1 2 K(x-y) \cdot (\nabla \log \bar{\rho}(x)-\nabla \log \bar{\rho}(y)).
\end{equation}
For simplicity, we also write 
\begin{equation*}
    \Phi_N (x_1,\cdots,x_N)=\frac{1}{N^2} \sum_{i,j=1}^N \phi(x_i,x_j).
\end{equation*}

In order to change the last term into an expectation with respect to the tensorized distribution $\bar\rho_N$, we recall the famous Donsker-Varadhan inequality as in \cite[Lemma 1]{jabin2018quantitative}: for any $\eta(t) > 0$, 
\begin{equation*}
    \int_{\mathbb{R}^{2N}}   \rho_N \Phi_N  \ud X^N \leq \frac{1}{\eta}\Big(H_N(\rho_N|\bar{\rho}_N)+\frac{1}{N}\log \int_{\R^{2N}} \bar{\rho}_N \exp(N \eta \Phi_N) \ud X^N \Big). 
\end{equation*}
Hence we have the further estimate:
\begin{align*}
    H_N(\rho_N|\bar{\rho}_N)(t) \leq &\, H_N(\rho_N^0|\bar{\rho}_N^0)+\int_0^t \frac{1}{\eta(s)} H_N(\rho_N|\bar{\rho}_N)(s) \ud s\\
    +&\, \frac{1}{N}\int_0^t \frac{1}{\eta(s)} \log \int_{\R^{2N}} \bar\rho_N \exp(\eta N\Phi_N) \ud X^N \ud s.
\end{align*}
Now it suffices to bound the exponential integral
\begin{equation*}
    \int_{\mathbb{R}^{2N}}  \bar{\rho}_N \exp(N\eta \Phi_N)  \ud X^N.
\end{equation*}
Once again we recall the large deviation type estimate \cite[Theorem 4]{jabin2018quantitative}.
\begin{thm}[Jabin-Wang]\label{deviation}
    Consider any $\phi(x,y)$ satisfying the canceling properties
    \begin{equation}\label{cancel}
        \int_{\mathbb{R}^2}  \phi(x,y) \bar{\rho}(x)  \ud x=0, \,  \mbox{for any\,  } y, \quad \int_{\mathbb{R}^2}  \phi(x,y) \bar{\rho}(y) \ud y=0,  \, \mbox{for any\,  } x,
    \end{equation}
    and \red{there exists a universal constant $C_{JW}=1600^2+36e^4$ such that}
    \begin{equation}\label{gamma}
        \gamma=\red{C_{JW}} \Big(\sup_{p \geq 1} \frac{\Vert \sup_y |\phi(\cdot, y)| \Vert_{L^p(\bar{\rho} \ud x)}}{p} \Big)^2 <1.
    \end{equation}
    Then we have
    \begin{equation*}
        \int_{\mathbb{R}^{2N}}  \bar{\rho}_N \exp\Big(\frac{1}{N} \sum_{i,j=1}^N \phi(x_i,x_j)\Big) \ud X^N \leq \frac{2}{1-\gamma} <\infty.
    \end{equation*}
\end{thm}

\red{We mention that the condition \eqref{gamma} automatically holds when $\phi \in L^\infty$ with small enough $L^\infty$ norm, which is exactly the case when working on the torus as in \cite{jabin2018quantitative}. Under this extra assumption, there is a simpler proof given by Lim-Lu-Nolen \cite{lim2020quantitative}, using the probabilistic method and martingale inequalities. However, we cannot expect $\phi \in L^\infty$ when working on the whole space due to the lack of a uniform positive lower bound on the density $\bar\rho$. This demonstrates the importance of the general condition \eqref{gamma}.
}

It is straightforward to verify that $\phi(x,y)$ defined in (\ref{phi}) satisfies the canceling condition (\ref{cancel}). Hence it suffices to bound the supremum in $p$ appearing in (\ref{gamma}) and choose $\eta$ to be small enough.

We first make the following critical observation. 
\begin{lemma}\label{psi}
    For any fixed $x$, the function $\red{\phi}(x,y)$ is $L^\infty$ in $y$ and can be estimated as 
    \begin{equation*}
        \sup_{y \in \mathbb{R}^2}  |\red{\phi}(x,y)| \leq \red{C_1'}(1+\sqrt{t}+|x|^2).
    \end{equation*}
    \red{Here the constant $C_1'$ is independent of $x$.}
\end{lemma}
\begin{proof}
    We consider respectively the three terms of $\red{\phi}(x,y)$. From Lemma \ref{kconw} and Theorem \ref{loggradient} we have
    \begin{equation*}
        |(K \ast \bar{\rho})(x) \cdot \nabla \log \bar{\rho} (x)| \leq C(1+|x|).
    \end{equation*}
    For the second term we bound by
    \begin{align*}
        & |(K \ast \bar{\rho})(y) \cdot \nabla \log \bar{\rho}(y)| \leq C(1+|y|)\Big|\int_{\mathbb{R}^2}  K(y-z)\bar{\rho}(z)\ud z \Big|\\
        &\leq C+C\int_{\mathbb{R}^2} \frac{|y|}{|z-y|}\bar{\rho}(z) \ud z\\
        &\red{\leq C+C\int_{|z-y| \geq 1} \frac{(1+|z|)\bar\rho(z)}{|z-y|} \ud z+C \int_{|z-y| \leq 1} \frac{|z-y|+|z|}{|z-y|} \bar\rho(z) \ud z}\\
        &\leq C+C\int_{|z-y| \geq 1} (1+|z|)\bar{\rho}(z) \ud z+\red{C\int_{\R^2} \bar\rho(z) \ud z+} C \Big(\sup_{s \in \mathbb{R}^2} |s|\bar{\rho}(s) \Big) \cdot \int_{|z-y| \leq 1} \frac{1}{|z-y|} \ud z\\
        &\leq C+C\int_{\mathbb{R}^2}  (1+|z|)\bar{\rho}(z) \ud z+C\Big(\sup_{s \in \mathbb{R}^2} |s| \bar{\rho}(s) \Big) \leq C(1+\sqrt{t}).
    \end{align*}
    \red{The last inequality comes from the Gaussian upper bound Lemma \ref{mostgauss}.}
    For the third term, we should deal with the cases $|y-x|\leq 1$ and $|y-x|\geq 1$ respectively. When $|y-x| \leq 1 $, we apply the mean-value theorem and Theorem \ref{loghessian} to obtain that 
    \begin{equation*}
        |K(x-y) \cdot (\nabla \log \bar{\rho}(x)-\nabla \log \bar{\rho}(y))| \leq C\sup_{|z-x| \leq 1} |\nabla^2 \log \bar{\rho}(z)| \leq C(1+|x|^2).
    \end{equation*}
    Otherwise we use Theorem \ref{loggradient}:
    \begin{equation*}
        |K(x-y) \cdot (\nabla \log \bar{\rho}(x)-\nabla \log \bar{\rho}(y))| \leq C\frac{1+|x|+|y|}{|y-x|} \red{\leq C\frac{1+|y-x|+2|x|}{|y-x|}} \leq C(1+|x|).
    \end{equation*}
    This completes the proof.
\end{proof}
Now recall from the remark \red{at the end of \cite[Section 1.3]{jabin2016mean}} that
\begin{equation*}
    \sup_{p \geq 1} \frac{\Vert f \Vert_{L^p(\bar{\rho}\ud x)}}{p}< \infty
\end{equation*}
is equivalent to \red{the condition} that there exists some $\lambda>0$ such that
\begin{equation}\label{lambda}
    \int_{\mathbb{R}^2}  e^{\lambda f} \bar{\rho} \ud x<\infty.
\end{equation}
Together with Lemma \ref{psi}, it suffices to check the exponential integrability condition (\ref{lambda}) for $f(x)=\red{C_1'}(1+\sqrt{t}+|x|^2)$. However, this is valid once we recall the Gaussian upper bound Lemma \ref{mostgauss}
\begin{equation*}
    \Bar{\rho}(x,t) \leq \frac{\red{C_2'}}{t \vee 1} \exp\red{\Big(}-\frac{|x|^2}{8t+\red{C_2'}}\red{\Big)}.
\end{equation*}
Taking $\lambda=\lambda(t)=\frac{1}{2\red{C_1'}(8t+\red{C_2'})}$, we can bound the exponential integral \eqref{lambda} by
\begin{equation}
    \int e^{\lambda(t)f} \bar\rho \ud x \red{\leq \frac{2C_2'}{1+t} \int e^{\frac{1+\sqrt{t}+|x|^2}{2(8t+C_2')}} e^{-\frac{|x|^2}{8t+C_2'}}} \leq \red{C_3'},
\end{equation}
\red{for some constant $C_3'$ that depends on $C_1,C_2,C_3, \|\bar\rho_0\|_{W^{2,\infty}}$}. However, by expanding the left-hand side, we obtain as in \cite[Section 1.3]{jabin2016mean} that
\begin{equation*}
    \int e^{\lambda(t)f}\bar\rho \ud x \geq \frac{\lambda^p}{p!} \|f\|_{L^p(\bar\rho \ud x)}^p \geq \frac{\lambda^p}{p^p} \|f\|_{L^p(\bar\rho \ud x)}^p.
\end{equation*}
Hence we have
\begin{equation*}
    \sup_{p \geq 1} \frac{\|f\|_{L^p(\bar\rho \ud x)}}{p} \leq \frac{1}{\lambda} \int e^{\lambda f}\bar\rho \ud x \red{\leq \frac{C_3'}{\lambda(t)} \leq C_4'(1+t)},
\end{equation*}
\red{for some constant $C_4'$}. We choose $\eta(t)$ small enough to satisfy the condition \eqref{gamma}. It suffices to let $$\eta(t)=\frac{1}{\red{C_5'}(1+t)}$$ \red{for some constant $C_5'$.} Plugging into the time evolution inequality of relative entropy and applying the Gr\"onwall's lemma, we come to the final result.

\section{Proof of Lemma \ref{cal}}
\label{appA}
Recall from (\ref{vor}) that
\begin{equation*}
    \partial_t \Bar{\rho}+(K \ast \Bar{\rho}) \cdot \nabla \Bar{\rho}=\Delta \Bar{\rho}.
\end{equation*}
By direct calculation,
\begin{align*}
    \partial_{x_i} \frac{|\nabla \Bar{\rho}|^2}{\Bar{\rho}}=\frac{2\nabla \Bar{\rho}}{\Bar{\rho}} \cdot \nabla \partial_{x_i} \Bar{\rho}-\frac{|\nabla \Bar{\rho}|^2}{\Bar{\rho}^2} \partial_{x_i}\Bar{\rho}.
\end{align*}
Differentiating both sides with respect to $x_i$ and summing up in $i$, we arrive at
\begin{equation*}
    \Delta \frac{|\nabla \Bar{\rho}|^2}{\Bar{\rho}}=\frac{2\nabla \Bar{\rho}}{\Bar{\rho}} \cdot \nabla \Delta \Bar{\rho}-\frac{2\nabla \Bar{\rho}}{\Bar{\rho}^2} \cdot \nabla^2 \Bar{\rho} \cdot \nabla \Bar{\rho}+\frac{2|\nabla^2 \Bar{\rho}|^2}{\Bar{\rho}}-\frac{|\nabla \Bar{\rho}|^2}{\Bar{\rho}^2} \Delta \Bar{\rho}+\frac{2|\nabla \Bar{\rho}|^4}{\Bar{\rho}^3}-\frac{\red{2\nabla^2 \bar\rho:(\nabla \bar\rho \otimes \nabla \bar\rho)}}{\Bar{\rho}^2}.
\end{equation*}
Note also that
\begin{equation*}
    \partial_t \frac{|\nabla \Bar{\rho}|^2}{\Bar{\rho}}=\frac{2\nabla \Bar{\rho}}{\Bar{\rho}} \cdot \nabla(\Delta \Bar{\rho}-(K \ast \Bar{\rho}) \cdot \nabla \Bar{\rho})-\frac{|\nabla \Bar{\rho}|^2}{\Bar{\rho}^2}(\Delta \Bar{\rho}-(K \ast \Bar{\rho}) \cdot \nabla \Bar{\rho}).
\end{equation*}
Hence by rearranging the terms we have
\begin{align*}
    (\partial_t -\Delta)\frac{|\nabla \Bar{\rho}|^2}{\Bar{\rho}}=&-\frac{2}{\Bar{\rho}}\red{\Big| \nabla^2} \Bar{\rho}-\frac{\red{\nabla} \Bar{\rho} \red{\otimes \nabla} \Bar{\rho}}{\Bar{\rho}}\red{\Big|}^2-\frac{2}{\Bar{\rho}}\red{\nabla^2 \bar\rho \otimes ((K \ast \bar\rho) \otimes \nabla \bar\rho)}\\
    &-\frac{2}{\Bar{\rho}} \red{\nabla(K \ast \bar\rho):(\nabla \bar\rho \otimes \nabla \bar\rho)}+\frac{|\nabla \Bar{\rho}|^2}{\Bar{\rho}^2}(K \ast \Bar{\rho}) \cdot \nabla \Bar{\rho}\\
    \leq &-(K \ast \Bar{\rho}) \cdot \nabla \frac{|\nabla \Bar{\rho}|^2}{\Bar{\rho}}-\frac{2}{\Bar{\rho}}\red{\nabla(K \ast \bar\rho):(\nabla \bar\rho \otimes \nabla \bar\rho)}\\
    \leq &-(K \ast \Bar{\rho}) \cdot \nabla \frac{|\nabla \Bar{\rho}|^2}{\Bar{\rho}}+\red{2A_2}\frac{|\nabla \Bar{\rho}|^2}{\Bar{\rho}},
\end{align*}
which proves (\ref{nablawsquare}).

Now we turn to the propagation of $\Bar{\rho} \log \Bar{\rho}$.
\begin{equation*}
    \partial_{x_i} (\Bar{\rho}\log \Bar{\rho})=\partial_{x_i} \Bar{\rho}+\log \Bar{\rho} \partial_{x_i}\Bar{\rho}.
\end{equation*}
Differentiating both sides with respect to $x_i$ and summing up in $i$, we arrive at
\begin{equation*}
    \Delta (\Bar{\rho}\log \Bar{\rho}) =\Delta \Bar{\rho}+\log \Bar{\rho} \Delta \Bar{\rho}+\frac{|\nabla \Bar{\rho}|^2}{\Bar{\rho}}.
\end{equation*}
Note also that
\begin{equation*}
    \partial_t (\Bar{\rho}\log \Bar{\rho})=(\Delta \Bar{\rho}-(K \ast \Bar{\rho})\cdot \nabla \Bar{\rho})+\log \Bar{\rho}(\Delta \Bar{\rho}-(K \ast \Bar{\rho} )\cdot \nabla \Bar{\rho}) .
\end{equation*}
Hence we have
\begin{align*}
    (\partial_t-\Delta)(\Bar{\rho}\log \Bar{\rho})&=-(K \ast \Bar{\rho}) \cdot \nabla \Bar{\rho}-\log \Bar{\rho} (K \ast \Bar{\rho}) \cdot \nabla \Bar{\rho}-\frac{|\nabla \Bar{\rho}|^2}{\Bar{\rho}}\\
    &=-(K \ast \Bar{\rho}) \cdot \nabla(\Bar{\rho}\log \Bar{\rho})-\frac{|\nabla \Bar{\rho}|^2}{\Bar{\rho}},
\end{align*}
which proves (\ref{wlogw}).

\section{Proof of Lemma \ref{cal2}}
\label{appB}
Recall from (\ref{vor}) that
\begin{equation*}
    \partial_t \Bar{\rho}+(K \ast \Bar{\rho}) \cdot \nabla \Bar{\rho}=\Delta \Bar{\rho}.
\end{equation*}
By direct calculation,
\begin{equation*}
    \partial_{x_i} \frac{|\nabla^2 \Bar{\rho}|^2}{\Bar{\rho}}=\frac{2}{\Bar{\rho}} \nabla^2 \Bar{\rho} :\nabla^2 \partial_{x_i}\Bar{\rho}-\frac{|\nabla^2 \Bar{\rho}|^2}{\Bar{\rho}^2 }\partial_{x_i}\Bar{\rho}.
\end{equation*}
Differentiating both sides with respect to $x_i$, and summing up in $i$, we arrive at
\begin{equation*}
    \Delta \frac{|\nabla^2 \Bar{\rho}|^2}{\Bar{\rho}}=\frac{2}{\Bar{\rho}} \nabla^2 \Bar{\rho} :\nabla^2 \Delta \Bar{\rho}+\frac{2}{\Bar{\rho}}|\nabla^3 \Bar{\rho}|^2-\frac{4}{\Bar{\rho}^2}\nabla^2 \Bar{\rho} \cdot \nabla^3 \Bar{\rho} \cdot \nabla \Bar{\rho}-\frac{|\nabla^2 \Bar{\rho}|^2}{\Bar{\rho}^2}\Delta \Bar{\rho}+\frac{2|\nabla^2 \Bar{\rho}|^2}{\Bar{\rho}^3}|\nabla \Bar{\rho}|^2.
\end{equation*}
Note also that
\begin{equation*}
    \partial_t \frac{|\nabla^2 \Bar{\rho}|^2}{\Bar{\rho}}=\frac{2}{\Bar{\rho}}\nabla^2 \Bar{\rho}: \nabla^2(\Delta \Bar{\rho}-(K \ast \Bar{\rho}) \cdot \nabla \Bar{\rho})-\frac{|\nabla^2 \Bar{\rho}|^2}{\Bar{\rho}^2}(\Delta \Bar{\rho}-(K \ast \Bar{\rho}) \cdot \nabla \Bar{\rho}).
\end{equation*}
Hence by rearranging the terms we have
\begin{align*}
    (\partial_t -\Delta) \frac{|\nabla^2 \Bar{\rho}|^2}{\Bar{\rho}}=&\, -\frac{2}{\Bar{\rho}} \red{\Big| \nabla^3} \Bar{\rho}-\frac{\red{\nabla} \Bar{\rho} \red{\otimes \nabla^2} \Bar{\rho}}{\Bar{\rho}} \red{\Big|}^2\\
    &\, -\frac{4}{\Bar{\rho}}\nabla^2 \Bar{\rho}: ((K \ast \nabla \Bar{\rho}) \cdot \nabla^2 \Bar{\rho})-\frac{2}{\Bar{\rho}} \nabla^2 \Bar{\rho} \cdot (K \ast \nabla^2 \Bar{\rho}) \cdot \nabla \Bar{\rho}\\
    &\, -\frac{2}{\Bar{\rho}}(K \ast \Bar{\rho}) \cdot \nabla^3 \Bar{\rho} \cdot \nabla^2 \Bar{\rho}+\frac{|\nabla^2 \Bar{\rho}|^2}{\Bar{\rho}^2}(K \ast \Bar{\rho}) \cdot \nabla \Bar{\rho}\\
    \leq&\, \red{\frac{5A_2}{t \vee 1}}\frac{|\nabla^2 \Bar{\rho}|^2}{\Bar{\rho}}+\red{\frac{A_2}{(t \vee 1)^2}}\frac{|\nabla \Bar{\rho}|^2}{\Bar{\rho}}-(K \ast \Bar{\rho}) \cdot \nabla \frac{|\nabla^2 \Bar{\rho}|^2}{\Bar{\rho}},
\end{align*}
which proves (\ref{hessianw}).

Now we turn to the propagation of $\Bar{\rho}(\log \Bar{\rho})^2$.
\begin{equation*}
    \partial_{x_i} (\Bar{\rho}(\log \Bar{\rho})^2)=\partial_{x_i} \Bar{\rho} (\log \Bar{\rho})^2+2\partial_{x_i}\Bar{\rho} \log \Bar{\rho}.
\end{equation*}
Differentiating both sides with respect to $x_i$, and summing up in $i$, we arrive at
\begin{equation*}
    \Delta (\Bar{\rho}(\log \Bar{\rho})^2) =\Delta \Bar{\rho}(\log \Bar{\rho})^2+2\frac{|\nabla \Bar{\rho}|^2}{\Bar{\rho}}\log \Bar{\rho}+2\Delta \Bar{\rho} \log \Bar{\rho}+2\frac{|\nabla \Bar{\rho}|^2}{\Bar{\rho}}.
\end{equation*}
Note also that
\begin{equation*}
    \partial_t (\Bar{\rho}(\log \Bar{\rho})^2)=(\Delta \Bar{\rho}-(K \ast \Bar{\rho}) \cdot \nabla \Bar{\rho})(\log \Bar{\rho})^2+2\log \Bar{\rho}(\Delta \Bar{\rho}-(K \ast \Bar{\rho}) \cdot \nabla \Bar{\rho}).
\end{equation*}
Hence we have
\begin{align*}
    (\partial_t -\Delta)(\Bar{\rho}(\log \Bar{\rho})^2)=-(K \ast \Bar{\rho}) \cdot \nabla (\Bar{\rho}(\log \Bar{\rho})^2)-\frac{2}{\Bar{\rho}}|\nabla \Bar{\rho}|^2(1+\log \Bar{\rho}),
\end{align*}
which proves (\ref{wlogwsquare}).

\medskip 

\noindent {\bf Acknowledgments.} X. F. and Z. W. are partially  supported by the National Key R\&D Program of China, Project Number 2021YFA1002800, NSFC grant No.12171009, Young Elite Scientist Sponsorship Program by China Association for Science and Technology (CAST) No. YESS20200028 and the Fundamental Research Funds for the Central Universities (the start-up fund), Peking University.

\noindent {\bf Declaration.} The authors declare that they have no conflict of interest.

\bibliography{ref}
\bibliographystyle{abbrv}

\end{document}